\documentclass[11pt,reqno]{amsart}
\usepackage{amsmath,amssymb,amscd,graphicx, color,tensor,hyperref,datetime}

\newcommand{\R}{{ \mathbb{R}  }}

\newcommand{\bke}[1]{\left( #1 \right)}
\newcommand{\bkt}[1]{\left[ #1 \right]}
\newcommand{\bket}[1]{\left\{ #1 \right\}}
\newcommand{\norm}[1]{\left\Vert #1 \right\Vert}
\newcommand{\abs}[1]{\left| #1 \right|}

\begin{document}
\bibliographystyle{plain}

\newtheorem{defn}{Definition}
\newtheorem{lemma}{Lemma}
\newtheorem{proposition}{Proposition}
\newtheorem{theorem}{Theorem}
\newtheorem{assumption}{Assumption}
\newtheorem{cor}{Corollary}
\newtheorem{remark}{Remark}
\numberwithin{equation}{section}

\newenvironment{pfthm1}{{\par\noindent\bf
           Proof of Theorem~\ref{THM1}. }}{\hfill\fbox{}\par\vspace{.2cm}}
\newenvironment{pfthm2}{{\par\noindent\bf
           Proof of Theorem~\ref{THM2}. }}{\hfill\fbox{}\par\vspace{.2cm}}
\newenvironment{pfthm3}{{\par\noindent\bf
           Proof of Theorem~\ref{THM2.5}. }}{\hfill\fbox{}\par\vspace{.2cm}}
\newenvironment{pfthm4}{{\par\noindent\bf
           Proof of Theorem~\ref{THM3}. }}{\hfill\fbox{}\par\vspace{.2cm}}


\title{Global well-posedness of   logarithmic Keller-Segel type systems}



\author{Jaewook Ahn}%
\address{Department of Mathematics, Dongguk University, Seoul 04620, Republic of Korea}%
\email{jaewookahn@dgu.ac.kr}
\author{Kyungkeun Kang}%
\address{School of Mathematics $\&$ Computing(Mathematics),  Yonsei University, Seoul 03722, Republic of Korea}%
\email{kkang@yonsei.ac.kr}
\author{Jihoon Lee}%
\address{Department of Mathematics, Chung-Ang University, Seoul 06974, Republic of Korea}%
\email{jhleepde@cau.ac.kr}

\begin{abstract}
  We consider a class of logarithmic Keller-Segel type systems   
modeling the spatio-temporal behavior of either  chemotactic cells or  criminal activities  in  spatial dimensions two and higher.
   Under certain assumptions on  parameter values and given functions,  the existence of  classical solutions  is established globally in time, provided that initial data are sufficiently regular. In particular, we enlarge the range of chemotatic sensitivity $\chi$, compared to known results, in case that spatial dimensions are between two and eight.
   In addition, we provide new type of small initial data to obtain global classical solution, which is also applicable to the urban crime model.  
  We discuss long-time asymptotic behaviors of solutions as well.
 \end{abstract}
  \maketitle

\section{Introduction}
The formation of high-density clusters  has been observed in the movement of  chemotactic cells \cite{KS71},   in the dynamics of self-gravitating particles \cite{P90},
and in the  time-evolution of     residential burglary  data \cite{SBBT08}.
  The study of spatio-temporal  dynamics of  such clusters is important since it can be used  to gain insight on how to enhance or suppress the formation of   clusters. 
In the real world, for example,  it can help   suppress the formation of   
   criminal hotspots observed in social problems.

One may use mathematical tools to analyze the spatio-temporal  dynamics of clusters. 
 In this paper,  we deal with the cross-diffusive system which describes either the  movement of   chemotactic cells \cite{KS71} or  the propagation of criminal activities \cite{SDPTBBC08}:
\begin{equation}\label{MODEL0}
 \left\{
\begin{array}{ll}
\partial_{t}u=\Delta u -\chi\nabla \cdot \bke{ u  \nabla \log v    }- \sigma u v +\varphi, \qquad& x\in \Omega,\,\, t>0, \\  
\partial_{t}v=\Delta v-v+uv^{\lambda}+\psi, \qquad& x\in \Omega,\,\, t>0, \\
\displaystyle \frac{\partial u}{\partial \nu} =  \frac{\partial v}{\partial \nu}=0, \qquad& x\in\partial \Omega,\,\, t>0, \\
u(x,0)=u_{0}(x),\,\, v(x,0)=v_{0}(x),\quad &x\in \Omega.
\end{array}
\right. 
\end{equation}
Here, $\Omega\subset\R^{d}$, $d\ge2$, is a smooth and bounded domain, $\nu$
 is  the unit outward normal vector to $\partial\Omega$,  $\chi$, $\sigma$  and $\lambda$ are  given non-negative parameters,   $\varphi$ and  $\psi$ are given non-negative  functions, and $u$ and $v$ are unknowns.
  
   The system \eqref{MODEL0}     represents urban crime model (we refer it to (UC) in the sequel), which is a generalized version of logarithmic Keller-Segel model. Namely, if $\sigma=\varphi=\psi=\lambda=0$, (UC) becomes 
  \[
  \partial_{t}u=\Delta u -\chi\nabla \cdot \bke{ u  \nabla \log v    }, \qquad
\partial_{t}v=\Delta v-v+u.
  \]
  From now on, we call it (KS), unless any confusion is to be expected.
  
In the following, we briefly introduce (KS), (UC) and related works.  Our main interests are their global classical solvability and    long-time asymptotics.

The system (KS) describes the biased movement of chemotactic cells toward  a higher chemoattractant concentration.  The unknown functions $u$ and $v$  denote  the density of chemotactic cells,  and the chemoattractant concentration, respectively.  As to global solvability in (KS), various thresholds of $\chi$ have been introduced. Winkler \cite{W11}  obtained the  global classical solvability for
 \begin{equation}\label{THS1}
 \chi<\sqrt{\frac{2}{d}},\quad  d\ge2,
 \end{equation}
 and later on, Lankeit \cite{L10}  relaxed the condition \eqref{THS1} for $d=2$ as $\chi<\chi_{0}\approx1.015$.
 In the case of the parabolic-elliptic counterpart, Nagai-Senba \cite{NS98} proved the existence of finite time blowup radial solutions for $\chi>\frac{2d}{d-2}$, $d\ge3$, and the existence of  global radial classical solutions for $\chi<\frac{2}{d-2}$, $d\ge3$. Furthermore, Fujie-Winkler-Yokota \cite{FWY15}
established global classical solvability for $\chi<\frac{2}{d}$, $d\ge2$, and later on, this  threshold for $d=2$ was enlarged to  infinity  by Fujie-Senba \cite{FS16}(see also \cite{FS18}).  For a generalized solution concept, we refer to \cite{AKL19,B99, B20, LW17,SW11}.  As to long-time asymptotics of (KS), Winkler-Yokota \cite{WY18} obtained   the asymptotic stability of constant steady states under \eqref{THS1} and the  smallness of the domain size  $|\Omega|$, and later on, Ahn \cite{A19} removed out the restriction on the domain size   by   assuming $\chi\le \frac{1}{2}$ and the convexity of $\Omega$. In the case of the parabolic-elliptic counterpart, qualitative properties of solution such as  eventual regularity and asymptotic behavior can be found in  \cite{AKL19}.

The system \eqref{MODEL0} with $\sigma>0$ and $\lambda=0$ or $1$ represents the  model (UC),  
 where    $u$ denotes the density of criminals, $v$  denotes  the attractiveness value, and given functions $\varphi$ and $\psi$  denote the density of additional criminals, and  the source of attractiveness, respectively. Taking into account two effects,  the broken window effect and the repeat near-repeat effect, the  spatio-temporal dynamic of criminal  occurrences is modeled in (UC). For more information on modeling, see \cite{SBB10,SBBT08,SDPTBBC08}.
 As to global solvability in (UC), Rodr\'iguez  \cite{R13} obtained the global classical solvability  for $(\chi,d,\lambda)=(1,2,0)$, and    Freitag \cite{F18}   
 obtained the existence of global   classical solutions for $\chi<\frac{2}{d}$, $d\ge2$, $\lambda=1$(see also \cite{SL19}).  
Later on,  Rodr\'iguez-Winkler \cite{RW19}  established the global classical solvability   for arbitrary   $\chi>0$ and $d=\lambda=1$(see also \cite{WWF14} for the case that $\varphi$, $\psi$ are constants).  
  Recently, Tao-Winkler \cite{TW20}  obtained the global classical solvability for arbitrary   $\chi>0$ and $(d,\lambda)=(2,1)$    under the smallness conditions on $\|u_0\|_{L^{2}(\Omega)}$, $\|\nabla \sqrt{v_0}\|_{L^{2}(\Omega)}$, $\|\varphi\|_{L^{\infty}(0,\infty;L^{2}(\Omega))}$, and $\|\nabla \sqrt{\psi}\|_{L^{\infty}(0,\infty;L^{2}(\Omega))}$. For a  generalized solution concept, Winkler \cite{W19} obtained the global existence of   radial   renormalized solution for    $\chi>0$, $d=2$, and $\lambda=1$.  
  As to long-time asymptotics of (UC), Shen-Li \cite{SL19} obtained   the asymptotic stability of constant steady states for $\chi<\frac{2}{d}$, $d\ge2$, $\lambda=1$ under the assumption that $\varphi\ge0$ and $\psi>0$ are spatial-temporal constants with certain smallness. Moreover, the long-time convergence results    $(u,v)\displaystyle\mathop{\rightarrow}(0,v_{\infty})$ in an appropriate sense have been obtained in   Rodr\'iguez-Winkler \cite{RW19}, Winkler \cite{W19},  and  Tao-Winkler \cite{TW20}, where $v_{\infty}$ denotes the solution to Neumann problem $-\Delta v_{\infty}+v_{\infty}=\psi_{\infty}$ with $\psi_{\infty}(x)=\displaystyle\lim_{t\rightarrow\infty}\psi(x,t)$.
 
As we mentioned above, global existence results for (KS) or (UC) with general $\chi>0$ are available only for certain generalized solution concepts \cite{LW17,SW11,W19},  or restricted to either $d=1$ \cite{RW19,WWF14} or  $d=2$ with small data \cite{TW20}.  We also note that   
the conditions on  $\chi$  for global classical solvability   in  \eqref{MODEL0} such as  \eqref{THS1}  were made when certain energy estimates were obtained.   
For example,   in \cite{F18,SL19,W11}, the condition \eqref{THS1}  is used to control
$\int_{\Omega}u^{p}v^{-q}$
with some $p>\frac{d}{2}$.

In this paper,  
we  develop a new approach  different from energy methods. Our main tool is the maximum principle, which is motivated and  modified  from those  used in Yang-Chen-Liu \cite[Theorem 2.3]{YCL01} and Winkler \cite{W14}.   
 More precisely, after   transforming the system   \eqref{MODEL0} into the single   equation for $
z=\frac{u}{v^{1-\lambda}}+(\lambda+\frac{\chi}{2})|\nabla\log v|^{2}$, 
\[
\partial_{t}z-\Delta z-2\nabla z\cdot\nabla \log v +F(u,v,\psi)
=G(u,v,\varphi,\psi)
\]
with \[
F(u,v,\psi)=(1-\lambda) \bke{\frac{u}{v^{1-\lambda}}}^{2}
 +\sigma  v\bke{ \frac{u}{v^{1-\lambda}} }+(1-\lambda)(\lambda+\chi)\bke{ \frac{u}{v^{1-\lambda}} }|\nabla \log v|^{2}
 \]
 \[+(2\lambda+ \chi )|D^{2}\log v|^{2}+ (1-\lambda)  \bke{ \frac{u}{v^{1-\lambda}} } \frac{\psi }{v}+(2\lambda+ \chi )\frac{\psi }{v}|\nabla \log v|^{2},
 \] and 
\[G(u,v,\varphi,\psi)=
-\chi\bke{\frac{u}{v^{1-\lambda}}}\Delta \log v+(1-\lambda) \bke{ \frac{u}{v^{1-\lambda}} }+\frac{\varphi }{v^{1-\lambda}} +\frac{(2\lambda+ \chi )}{2v}\nabla \log v\cdot \nabla \psi,
\]
we  control the right-hand side $G(u,v,\varphi,\psi)$ by using the terms in  $F(u,v,\psi)$, and    apply the  maximum principle argument. Here, the     challenging term in $ G(u,v,\varphi,\psi)$ is 
\[
-\chi\bke{\frac{u}{v^{1-\lambda}}}\Delta \log v
\]
but as we will show in \eqref{UVEQ3}--\eqref{UVEQ4},
 we can control this term if $\chi\le \chi_{d,\lambda}$ with $\chi_{d,\lambda}$ defined in \eqref{CHIDLAM} (Theorem~\ref{THM1}).
 In particular, our global solvability result covers (KS) with $\chi=\chi_{2,0}=2$, $d=2$ for large data. In the case of (UC)  with $\chi=2>0=\chi_{2,1}$, $d=2$, $\lambda=1$, our global solvability  result requires $\sigma>0$ and additional smallness conditions on $u_{0}$, $\nabla \log v_{0}$, $\varphi$, and $\nabla \sqrt{\psi}$.
 
  The first   goal of this paper is  the global well-posedness result for large data.
The  initial data $u_{0}$, $v_{0}$, and given functions $\varphi$, $\psi$ of  \eqref{MODEL0} are supposed to  satisfy the following:
\begin{assumption}\label{MODEL012} $u_{0}$, $v_{0}$,  $\varphi$, and $\psi$ are  all  non-negative and  
\[
\displaystyle (u_{0},v_{0})\!\in\! \bke{\mathcal{C}^{2+\alpha}(\overline{\Omega})}^{2} \mbox{ for some }\alpha\in(0,1),\qquad \min_{\overline{\Omega}} v_{0}>0, 
\]
\[
\frac{\partial u_{0}}{\partial \nu}=\frac{\partial v_{0}}{\partial \nu}=0,\qquad x\in\partial\Omega,
\]
\[
    \varphi \! \in\!  \mathcal{C}^{1}(\overline{\Omega}\times[0,\infty))\!\cap\! L^{\infty}(\Omega\times(0,\infty)),\quad
  \psi  \! \in\!  \mathcal{C}^{2}(\overline{\Omega}\times[0,\infty))\!\cap\! L^{\infty}((0,\infty);W^{1,\infty}(\Omega)).\]
\end{assumption}

Although  $\lambda=0$ or $1$ is the most interesting case, we consider $ \lambda\in[0,1]$ to see 
how $\lambda$ affects the threshold of $\chi$.
 For convenience, we denote 
 \begin{equation}\label{CHIDLAM}
 \chi_{d,\lambda}:=\frac{2}{d}\bke{  1-\lambda +\sqrt{  2d\lambda(1-\lambda)+(1-\lambda)^2   }            }.
 \end{equation}
Our first main result, which  is for large data and $\lambda\in[0,1)$,    reads as follows:  
 \begin{theorem}\label{THM1}
Let $\Omega$ be a smooth, bounded and convex domain of $\R^{d}$, $d\ge2$. Suppose that $(u_0,v_0,\varphi,\psi)$ satisfies Assumption~\ref{MODEL012}, and let $\sigma\ge0$, $0\le\lambda<1$. If  
$
\chi\le \chi_{d,\lambda}$,
then  \eqref{MODEL0}  possess a unique  non-negative smooth solution $(u,v)$ globally in time in the class
\begin{equation}\label{THM1_1}
\begin{array}{ll}
u,v\in    \bke{\mathcal{C}^{2+\alpha,1+\frac{\alpha}{2}}(\overline{\Omega}\times[0,\infty))}^{2}.
\end{array}
\end{equation}
Moreover, if we  further assume that  $\chi< \chi_{d,\lambda}$   and  one of the followings:
\begin{equation}\label{THM1ASS1}
 \left.
\begin{array}{ll}
\bullet&  \sigma>0, \,\,\displaystyle\inf_{t\ge0}  \int_{\Omega}\psi(\cdot,t)>0, \\  
\bullet& (\varphi(x,t),\psi(x,t))=(0,b(t)) \quad \mbox{ for all }\,\, x\in \Omega,\, t>0,
\end{array}
\right. 
\end{equation}
 then 
 there exists a constant $C>0$ independent of $t$ such that
\begin{equation}\label{THM1_2}
\norm{u(\cdot,t)}_{ L^{\infty}(\Omega) }+\norm{  v(\cdot,t)}_{ L^{\infty}(\Omega) }+\norm{\nabla \log v(\cdot,t)}_{  L^{\infty}(\Omega) }\le C\quad\mbox{ for all } t>0.
\end{equation}
\end{theorem}
\begin{remark}
According to Theorem~\ref{THM1}, {\rm (KS)}  with $
\chi\le \chi_{d,0}=\frac{4}{d}$ possess a unique  non-negative smooth solution $(u,v)$ globally in time in the class \eqref{THM1_1}. In particular, if 
the strict inequality $
\chi< \chi_{d,0}$ holds, then  there exists a constant $C>0$ satisfying
 \eqref{THM1_2}.
Thus, in the case that $2\le d\le 8$,  Theorem~\ref{THM1} improves the previous results of Winkler  \cite{W11} and   Lankeit \cite{L10} since $
\chi_{d,0}=\frac{4}{d}\ge\sqrt{ \frac{2}{d} }$ and $\chi_{2,0}=2>1.016$.   Theorem~\ref{THM1} is also an improvement of  the  result in \cite{R13} for {\rm (UC)}  with $(\lambda,d)=(0,2)$. 
\end{remark}
 
 From now on, we assume that $\psi$ is strictly positive and satisfies 
 \begin{equation}\label{ASSTHM20}
0<\eta:= \displaystyle\inf_{x\in\overline{\Omega},\\s\ge0} \psi(x,s) \le \displaystyle\min_{\overline{\Omega}} v_0.
\end{equation} 
The second   goal of this paper is to establish  the global well-posedness result for \eqref{MODEL0}  with $\sigma>0$ under certain smallness conditions.  More precisely, if either $\chi$ is large or $\sigma\eta$ is small, 
then smallness of given data implies global existence of classical solutions. In case that $\sigma\eta$ is large compared to $\chi$, it is not necessary that given data are small, but they are required to be bounded by certain numbers. This seems to be due to damping effects of parameters $\sigma$ and $\eta$. 

We first state the case $\lambda=1$.
\begin{theorem}\label{THM2}
Let $\Omega$ be a smooth, bounded and convex domain of $\R^{d}$, $d\ge2$, and let  $\sigma>0$,    $\lambda=1$. 
Suppose that $(u_0,v_0,\varphi,\psi)$ satisfies   Assumption~\ref{MODEL012} and \eqref{ASSTHM20}.
We set    
\[
\displaystyle\delta=\min\bket{\frac{1}{2},  \frac{ \sigma\eta(4+2\chi )  }{ d\chi^{2}} },\qquad \mu=\max\{ \sigma^{-1},\,    \|v_0\|_{L^{\infty}(\Omega)},\,\|\psi\|_{L^{\infty}(\Omega\times(0,\infty))}   \}.
\]
Assume further that  $(u_0,v_0,\varphi,\psi)$    satisfies 
\begin{equation}\label{ASSTHM21}
\Bigr{\|}u_{0}+(1+\frac{\chi}{2})|\nabla \log v_{0}|^{2}\Bigr{\|}_{L^{\infty}(\Omega)}<\delta,
\end{equation}
and
\begin{equation}\label{ASSTHM22}
\Bigr{\|} \varphi  +\frac{4+2\chi}{\eta } |\nabla \sqrt{\psi}|^{2} \Bigr{\|}_{L^{\infty}(\Omega\times(0,\infty))}  <\frac{\eta\delta }{2\mu}.
\end{equation}
Then, \eqref{MODEL0}  possess a unique  non-negative smooth solution $(u,v)$ globally in time in the class \eqref{THM1_1}. Moreover, there exists a constant $C>0$ independent of $t$ satisfying \eqref{THM1_2}. 
\end{theorem}
\begin{remark}
For   $d\ge2$, $\sigma>0$, and $\lambda=1$, a quadruple $(u_{0}, v_{0}, \varphi, \psi)$ satisfies  \eqref{ASSTHM20}--\eqref{ASSTHM22} if  it is   sufficiently close to a constant vector $(0,a,0,b)$  with  $a\ge b>0$. 
\end{remark} 

The proof of Theorem~\ref{THM2} is also valid with some modifications to the case $\lambda\in[0,1)$, which seems to be of independent interest, because it shows how $\lambda$ affects values of $\delta$, $\mu$, $\sigma$, and $\eta$.
 Since the global well-posedness for $
\chi\le \chi_{d,\lambda}$ with general large data is resolved   by Theorem~\ref{THM1},
we only  treat  the case $
\chi> \chi_{d,\lambda}$.

\begin{theorem}\label{THM2.5}
Let $\Omega$ be a smooth, bounded and convex domain of $\R^{d}$, $d\ge2$, and let $\sigma>0$,  $0\le\lambda<1$, and $
\chi> \chi_{d,\lambda}$  with $\chi_{d,\lambda}$  defined in \eqref{CHIDLAM}. 
Suppose that $(u_0,v_0,\varphi,\psi)$ satisfies   Assumption~\ref{MODEL012} and \eqref{ASSTHM20}. We set \[
 \delta_{0}=\displaystyle\min\bket{\frac{1}{2},\,\,  \frac{\sigma \eta}{2}\Bigr{(}  \frac{d\chi^2}{8\lambda+4\chi}-(1-\lambda)    \Bigr{)}^{-1} },
 \]
 \[ 
 \mu_{0}=\max\{ 2\sigma^{-1},\,    \|v_0\|_{L^{\infty}(\Omega)},\,\|\psi\|_{L^{\infty}(\Omega\times(0,\infty))}   \}.
 \] 
 Assume further that
    $(u_0,v_0,\varphi,\psi)$   satisfies 
 \[
 \sigma \eta \ge 4(1-\lambda),
 \]
\begin{equation}\label{THM3_2}
\Bigr{\|}u_{0}+(\lambda+\frac{\chi}{2})|\nabla \log v_{0}|^{2}\Bigr{\|}_{L^{\infty}(\Omega)}<\delta_{0},
\end{equation}
and
\begin{equation}\label{THM3_3}
\Bigr{\|} \frac{\varphi }{\eta^{1-\lambda}} + \frac{4\lambda+2\chi}{\eta } |\nabla \sqrt{\psi}|^{2} \Bigr{\|}_{L^{\infty}(\Omega\times(0,\infty))}   < \frac{\eta\delta_{0} }{2 \mu_{0}}.
\end{equation}
Then, \eqref{MODEL0}  possess a unique  non-negative smooth solution $(u,v)$ globally in time in the class \eqref{THM1_1}. Moreover, there exists a constant $C>0$ independent of $t$ satisfying \eqref{THM1_2}.   
\end{theorem} 

Finally, we state the long time behavior result for $\sigma>0$.
 \begin{theorem}\label{THM3}
Let $\Omega$ be a smooth, bounded and convex domain of $\R^{d}$, $d\ge2$, and let $\sigma>0$, $0\le \lambda \le1$. Suppose that $(u, v)$ is a unique  global classical solution to  \eqref{MODEL0}
satisfying \eqref{THM1_2}. If 
\begin{equation}\label{LOWV2}
   \eta_{0} \le \min_{\overline{\Omega}}  v(\cdot, t) \quad\mbox{ for all } t>0, 
 \end{equation}
where $\eta_{0}$ is a positive constant independent of $t$, and 
 \begin{equation}\label{ASSPHIPSI}
 \int_{0}^{\infty}\int_{\Omega}\varphi(x,t)dxdt+
 \int_{0}^{\infty}\int_{\Omega}|\psi(x,t)-\psi_{\infty}(x)|^{2}dxdt<\infty
  \end{equation}
for some $\psi_{\infty}\in \mathcal{C}(\overline{\Omega})$ with $\displaystyle\min_{\overline{\Omega}} \psi_{\infty} >0$, then $u$ and $v$ satisfy
\[
\|u(\cdot,t)\|_{L^{\infty}(\Omega)}\rightarrow 0,\qquad \|v(\cdot,t)-v_{\infty}\|_{L^{\infty}(\Omega)}\rightarrow 0\qquad \mbox{ as }  t\rightarrow\infty,
\]
where $ v_{\infty}$ denotes the solution for  
\begin{equation}\label{VINFMODEL}
 \left\{
\begin{array}{ll}
v_{\infty}-\Delta v_{\infty}=\psi_{\infty}, \qquad &x\in\Omega,\\
\frac{\partial v_{\infty}}{\partial \nu}=0,   &x\in\partial \Omega.
\end{array}
\right.
\end{equation}
 \end{theorem}
\begin{remark}
We remark that the   property \eqref{THM1_2} holds for solutions in  Theorem~\ref{THM2}, Theorem~\ref{THM2.5}, and   Theorem~\ref{THM1} with $\chi<\chi_{d,\lambda}$ and \eqref{THM1ASS1}. Thus, such solutions satisfy asymptotics  in Theorem~\ref{THM3} under the assumptions \eqref{LOWV2}--\eqref{ASSPHIPSI}.
\end{remark}
The remaining part of this paper is organized as follows: in Section~\ref{SEC2}, we deal with a local well-posedness result involving lower estimate for $v$; in Section~\ref{SEC3}, we prove the global well-posedness for large data(Theorem~\ref{THM1});  
 in Section~\ref{SEC4}, we prove the global well-posedness for small data(Theorem~\ref{THM2}   and Theorem~\ref{THM2.5}); in Section~\ref{SEC5}, we prove a long time asymptotics(Theorem~\ref{THM3}). Throughout this paper, $C$ will denote a generic constant that may change from line to line.

\section{Local well-posedness}\label{SEC2}

In this section,  the local well-posedness   for \eqref{MODEL0} is obtained. Note that
  \eqref{MODEL0}  has   singular structure $\frac{1}{v}$ in the drift term $\nabla \cdot \bke{ u  \frac{\nabla v}{v}  }$ but due to the positivity of $v_0$ in Assumption~\ref{MODEL012}, such a   singularity  does not occur in a finite time unless blow-up occurs.
\begin{lemma}\label{LEM_LE}
Let $\Omega$ be a smooth,  bounded and convex domain of $\R^{d}$, $d\ge2$.  Suppose that $(u_0,v_0,\varphi,\psi)$ satisfies Assumption~\ref{MODEL012}, and let $\sigma\ge0$, $0\le \lambda\le1$, and $q>d$. Then,
there exists the maximal time of existence, $T_{\rm max}\le \infty$, such that a unique non-negative   solution $(u,v)$ of  \eqref{MODEL0}  exists and satisfies
\[
u,v\in \bke{  \mathcal{C}^{2+\alpha, 1+\frac{\alpha}{2}}(\overline{\Omega}\times[0,T_{\rm max}))}^{2},
\]
\begin{equation}\label{LEM_LE1E}
 e^{-t}\min_{\overline{\Omega}} v_{0} \le  \min_{\overline{\Omega}} v(\cdot,t)  \qquad \mbox{ for }\, t<T_{\rm max}
\end{equation}
and     
\begin{equation}\label{LEM_LE2E}
\mbox{either }\, T_{\rm max}=\infty \,\,\mbox{ or }\,\, \lim_{t\nearrow T_{\rm max}}   ( \|u(\cdot,t)\|_{L^{\infty}( \Omega )} +\|v(\cdot,t)\|_{W^{1,q}( \Omega )} )=\infty.
\end{equation}
\end{lemma}

Note that except for the regularity near $t=0$,
$u,v\in \bke{  \mathcal{C}^{2+\alpha, 1+\frac{\alpha}{2}}(\overline{\Omega}\times[0,T_{\rm max}))}^{2}$,
  Lemma~\ref{LEM_LE} with  $\sigma=\lambda=1$  was obtained in    \cite[Lemma~1--Lemma~2]{F18}. Since $u_0,v_0\in \bke{  \mathcal{C}^{2+\alpha}(\overline{\Omega})}^{2}$, we have $u,v\in \bke{  \mathcal{C}^{2+\alpha, 1+\frac{\alpha}{2}}(\overline{\Omega}\times[0,T_{\rm max}))}^{2}$   by Schauder's estimate(see e.g., \cite{LSU88}).   We  omit the proof of Lemma~\ref{LEM_LE} since its generalization to other cases  is rather straightforward(see also \cite{RB10,W11}).

 The temporal   bound \eqref{LEM_LE1E} can be replaced by uniform one
if we further assume that
\begin{equation}\label{LOWV1}
\mbox{ either }\quad  \displaystyle\inf_{t\ge0}   \int_{\Omega}\psi(\cdot,t) >0\quad \mbox{ or }\quad    \sigma=0\le\lambda<1,\,\, \|u_{0}\|_{L^{1}(\Omega)}>0.
 \end{equation}
\begin{lemma}\label{UNIFLOWV}
Let the same assumptions as in Lemma~\ref{LEM_LE} be satisfied.
  Let $(u,v)$ be the solution to \eqref{MODEL0}  given by Lemma~\ref{LEM_LE}. If \eqref{LOWV1} holds, then there exists a positive constant $\eta_{1}$ independent of $t$ such that
\begin{equation}\label{LEM2_LE}
\eta_{1}\le \min_{\overline{\Omega}} v(\cdot, t) \quad \mbox{ for all }\, t<T_{\rm max}.
 \end{equation}
In particular, if  $\displaystyle  \inf_{x\in\overline{\Omega},\\s\ge0} \psi(x,s)  >0$,  then there exists a positive constant 
\[
\eta_{2}=\min \{    \min_{\overline{\Omega}} v_0,\, \inf_{x\in\overline{\Omega},\\s\ge0} \psi(x,s)          \}
\] independent of $t$ satisfying
 \begin{equation}\label{LEM2_LE2}
\eta_{2}\le \min_{\overline{\Omega}} v(\cdot, t) \quad \mbox{ for all }\, t<T_{\rm max}.
 \end{equation}
 \end{lemma}
\begin{proof} 
Let  $ t<T_{\rm max}$ and denote ${\rm diam }\,\Omega:=\displaystyle\max_{x,y\in\overline{\Omega}}|x-y|$. As in the proof of \cite[Lemma 2.2]{F15}, we  use the representation formula of $v$ and $e^{t(\Delta-1)}v_0\ge0$  to obtain that
\[
v(t)=e^{t(\Delta-1)}v_0+\int_{0}^{t}e^{(t-s)(\Delta-1)}(uv^{\lambda} +\psi)(s) ds
\ge \int_{0}^{t}e^{(t-s)(\Delta-1)}(uv^{\lambda} +\psi)(s) ds 
\]
\begin{equation}\label{LEM4}
\ge \int_{0}^{t} \frac{1}{(4\pi(t-s))^{\frac{d}{2}}}e^{-(  \frac{({\rm diam }\,\Omega)^{2}}{4(t-s)} +      (t-s))}\bke{\int_{\Omega}uv^{\lambda}(x,s) +\psi(x,s)dx}ds.
\end{equation}
Now, using \eqref{LEM4}, we show \eqref{LEM2_LE}. We treat two cases in \eqref{LOWV1} separately. 

$\bullet$   ${\displaystyle\inf_{t\ge0}}  \int_{\Omega}\psi(\cdot,t)>0$ case.\\ 
 The right-hand side  of \eqref{LEM4} has a lower bound 
\begin{align*}
 \int_{0}^{t} &\frac{1}{(4\pi(t-s))^{\frac{d}{2}}}e^{-(  \frac{({\rm diam }\,\Omega)^{2}}{4(t-s)} +      (t-s))}\bke{\int_{\Omega}uv^{\lambda}(x,s) +\psi(x,s)dx}ds
\\
  &\ge\int_{0}^{t} \frac{1}{(4\pi(t-s))^{\frac{d}{2}}}e^{-(  \frac{({\rm diam }\,\Omega)^{2}}{4(t-s)} +      (t-s))}\bke{\int_{\Omega} \psi(x,s)dx}ds
\\
 &\ge \bke{{\displaystyle\inf_{s\ge0}}  \int_{\Omega}\psi(\cdot,s)}\int_{0}^{t} \frac{1}{(4\pi r)^{\frac{d}{2}}}e^{-(  \frac{({\rm diam }\,\Omega)^{2}}{4r} +      r)} dr,
\end{align*}
which is a monotone increasing function of $t$  with initial value zero. Combining it with  the decreasing lower bound in \eqref{LEM_LE1E}, we have a desired bound \eqref{LEM2_LE}.

$\bullet$  $\sigma=0\le\lambda<1$, $\|u_{0}\|_{L^{1}(\Omega)}>0$ case.
\\
 Integrating  $u$-equation   over $\Omega$,  we first note that  
$
\int_{\Omega}u(\cdot,t)\ge\int_{\Omega}u_0
$ for   $ t<T_{\rm max}$. 
We compute a  lower bound of the  right-hand side of  \eqref{LEM4} as 
\begin{align*}
 \int_{0}^{t} &\frac{1}{(4\pi(t-s))^{\frac{d}{2}}}e^{-(  \frac{({\rm diam }\,\Omega)^{2}}{4(t-s)} +      (t-s))}\bke{\int_{\Omega}uv^{\lambda}(x,s) +\psi(x,s)dx}ds
\\
 &\ge\int_{0}^{t} \frac{1}{(4\pi(t-s))^{\frac{d}{2}}}e^{-(  \frac{({\rm diam }\,\Omega)^{2}}{4(t-s)} +      (t-s))}\bke{\int_{\Omega} uv^{\lambda}(x,s) dx}ds
\\
&\ge
  \int_{0}^{t} \frac{1}{(4\pi(t-s))^{\frac{d}{2}}}e^{-(  \frac{({\rm diam }\,\Omega)^{2}}{4(t-s)} +      (t-s))}\bke{\int_{\Omega} u(x,s)dx  }(\min_{\overline{\Omega}} v(\cdot,s) )^{\lambda}ds
\\
 &\ge 
  \int_{0}^{t} \frac{1}{(4\pi r)^{\frac{d}{2}}}e^{-(  \frac{({\rm diam }\,\Omega)^{2}}{4r } +      r)}dr\bke{\int_{\Omega} u_0   }(\inf_{  s\le t} \min_{\overline{\Omega}} v(\cdot,s)  )^{\lambda}.
\end{align*}
By \eqref{LEM4}, it follows that 
   \begin{equation}\label{LEM4_1}
  \min_{\overline{\Omega}} v(\cdot,t) \ge   \int_{0}^{t} \frac{1}{(4\pi r)^{\frac{d}{2}}}e^{-(  \frac{({\rm diam }\,\Omega)^{2}}{4r } +      r)}dr\bke{\int_{\Omega} u_0   }(\inf_{  s\le t} \min_{\overline{\Omega}} v(\cdot,s)  )^{\lambda}.
   \end{equation}
Note that there exists a unique $\tau>0$ satisfying
$
\displaystyle e^{-\tau}\min_{\overline{\Omega}} v_{0} =f(\tau)$, where
 \[
 f(t):=\bkt{
  \int_{0}^{t} \frac{1}{(4\pi r)^{\frac{d}{2}}}e^{-(  \frac{({\rm diam }\,\Omega)^{2}}{4r } +      r)}dr\bke{\int_{\Omega} u_0   } }^{\frac{1}{1-\lambda}},\qquad t>0.
 \]  
 Indeed,  $\displaystyle e^{-t}\min_{\overline{\Omega}} v_{0}  $ is the monotone decreasing function of $t$ approaching zero for large $t$, and $f(t)$
 is the monotone increasing function of $t$  with initial value zero.   Note also from  \eqref{LEM_LE1E} that
 $\displaystyle \min_{\overline{\Omega}} v(\cdot, t) \ge e^{-\tau}\min_{\overline{\Omega}} v_{0} $
 for all $ t\le\tau$. Now, let $\tilde{\eta} \in(0,e^{-\tau}\displaystyle\min_{\overline{\Omega}} v_{0} )$ and suppose that 
\[
 \displaystyle \min_{\overline{\Omega}} v(\cdot, t_{1}) =\tilde{\eta}  
\]
  for the first time $t_{1}>\tau$.  By \eqref{LEM4_1}, we have
\begin{align*}
  \min_{\overline{\Omega}} v(\cdot,t_{1}) &\ge   \int_{0}^{t_{1}} \frac{1}{(4\pi r)^{\frac{d}{2}}}e^{-(  \frac{({\rm diam }\,\Omega)^{2}}{4r } +      r)}dr\bke{\int_{\Omega} u_0   }(\inf_{  s\le t_{1}} \min_{\overline{\Omega}} v(\cdot,s)  )^{\lambda}
\\
&= \int_{0}^{t_{1}} \frac{1}{(4\pi r)^{\frac{d}{2}}}e^{-(  \frac{({\rm diam }\,\Omega)^{2}}{4r } +      r)}dr\bke{\int_{\Omega} u_0   }( \min_{\overline{\Omega}} v(\cdot,t_{1}) )^{\lambda}
\\
&\ge \int_{0}^{\tau} \frac{1}{(4\pi r)^{\frac{d}{2}}}e^{-(  \frac{({\rm diam }\,\Omega)^{2}}{4r } +      r)}dr\bke{\int_{\Omega} u_0   }( \min_{\overline{\Omega}} v(\cdot,t_{1}) )^{\lambda},
\end{align*}
namely,
\[
\tilde{\eta}\ge f(\tau)= e^{-\tau}\min_{\overline{\Omega}} v_{0} ,\]
which leads to contradiction to  $\tilde{\eta} \in(0,e^{-\tau}\displaystyle\min_{\overline{\Omega}} v_{0} )$. Therefore,  
 $\displaystyle \min_{\overline{\Omega}} v(\cdot, t) \ge e^{-\tau}\min_{\overline{\Omega}} v_{0} $
for all $t>0$ and thus, \eqref{LEM2_LE} is obtained. 

Next,   \eqref{LEM2_LE2} is a direct consequence of the maximum principle applied to $v$-equation with   $\displaystyle  \inf_{x\in\overline{\Omega},\\s\ge0} \psi(x,s)  >0$. This completes the proof.
\end{proof}

\section{Global well-posedness for large data}\label{SEC3}
 In this section, we prove Theorem~\ref{THM1}.   Using the   maximum principle argument, we
 obtain an estimate for  $\frac{u}{v^{1-\lambda}}$ and $\nabla \log v$. 
 \begin{proposition}\label{PROP1}
 Let $\Omega$ be a smooth, bounded and convex domain of $\R^{d}$, $d\ge2$, and let $\sigma\ge0$, $0\le\lambda<1$. Suppose that $(u_0,v_0,\varphi,\psi)$ satisfies Assumption~\ref{MODEL012},  and    $(u,v)$ is a unique solution to \eqref{MODEL0}  given by Lemma~\ref{LEM_LE}. Then, there exists a      constants $C>0$  independent of   $t$  such that we have the following:
  \begin{itemize}
 \item[(i)] If $\chi\le \chi_{d,\lambda}$, then
 \begin{equation}\label{PR1_1}
  \norm{\frac{u}{v^{1-\lambda}}(\cdot,t)}_{L^{\infty}(\Omega)}+  \norm{  \nabla \log v (\cdot,t)  }_{L^{\infty}(\Omega)}^2  \le C e^{2 t} \qquad \mbox{ for }\, t<T_{\rm max},
  \end{equation}
 \item[(ii)] If $\chi< \chi_{d,\lambda}$, $\sigma>0$, and  ${\displaystyle\inf_{s\ge0}}  \int_{\Omega}\psi(\cdot,s) >0$,  then
 \begin{equation}\label{PR1_3}
  \norm{\frac{u}{v^{1-\lambda}}(\cdot,t)}_{  L^{\infty}(\Omega) }+ \norm{ \nabla \log v(\cdot,t)}_{ L^{\infty}(\Omega) }+ \norm{v(\cdot,t)}_{ L^{\infty}(\Omega) }\le C,
\end{equation}
 \item[(iii)] If $\chi< \chi_{d,\lambda}$ and $(\varphi(x,s),\psi(x,s))=(0,b(s))$ for all $x\in \Omega$, $s>0$, then \begin{equation}\label{PR1_2}
  \norm{\frac{u}{v^{1-\lambda}}(\cdot,t)}_{ L^{\infty}(\Omega) }+ \norm{\nabla \log v(\cdot,t)}_{  L^{\infty}(\Omega) }\le C.
\end{equation}
 \end{itemize}
 \end{proposition}
\begin{proof}
From \eqref{MODEL0}, we  obtain
\begin{align}
\begin{aligned}\label{UEQ2}
\partial_{t}&\bke{ \frac{u}{v^{1-\lambda}} }-\Delta \bke{ \frac{u}{v^{1-\lambda}} }+(1-\lambda)(\lambda+\chi)\bke{ \frac{u}{v^{1-\lambda}} }|\nabla \log v|^{2}
\\
&+(1-\lambda) \bke{\frac{u}{v^{1-\lambda}}}^{2}+\sigma  v\bke{ \frac{u}{v^{1-\lambda}} }+(1-\lambda)  \bke{ \frac{u}{v^{1-\lambda}} } \frac{\psi  }{v}
\\
&=(2-2\lambda-\chi)\nabla \bke{ \frac{u}{v^{1-\lambda}} }\cdot \nabla \log v -\chi\bke{\frac{u}{v^{1-\lambda}}}\Delta \log v+(1-\lambda) \bke{ \frac{u}{v^{1-\lambda}} }+\frac{\varphi }{v^{1-\lambda}}.
\end{aligned}
\end{align}
From $\eqref{MODEL0}_{2}$, we  derive   
\begin{align}
\begin{aligned}\label{VEQ2}
\partial_{t}&|\nabla \log v|^{2}-\Delta |\nabla \log v|^{2}+2|D^{2}\log v|^{2}+2\frac{\psi }{v}|\nabla \log v|^{2}
\\
& =2\nabla |\nabla \log v|^{2}\cdot\nabla \log v +2\nabla\bke{ \frac{u}{v^{1-\lambda}}}\cdot\nabla \log v +\frac{2}{v}\nabla \log v\cdot \nabla \psi. 
\end{aligned}
\end{align}
Denote $z:=\frac{u}{v^{1-\lambda}}+\theta|\nabla \log v|^{2}$, where $\theta=\lambda+\frac{\chi}{2}$. Note that 
\begin{equation}\label{ZBCCON}
\frac{\partial z}{\partial\nu}\le0 \qquad\mbox{ on } \partial\Omega,
\end{equation}
which is due to $ \frac{\partial u}{\partial \nu} =  \frac{\partial v}{\partial \nu}=0$ and the convexity of  $\Omega$.
Adding $\eqref{UEQ2}$ and $\eqref{VEQ2}\times\theta $, we have
\[
\partial_{t}z-\Delta z-2\nabla z\cdot\nabla \log v 
+(1-\lambda)(\lambda+\chi)\bke{ \frac{u}{v^{1-\lambda}} }|\nabla \log v|^{2}+(1-\lambda) \bke{\frac{u}{v^{1-\lambda}}}^{2}
\]
 \[
 +\sigma  v\bke{ \frac{u}{v^{1-\lambda}} }+(1-\lambda)  \bke{ \frac{u}{v^{1-\lambda}} } \frac{\psi }{v}+2\theta|D^{2}\log v|^{2}+2\theta\frac{\psi }{v}|\nabla \log v|^{2}
 \]
\begin{equation}\label{UVEQ2}
=
-\chi\bke{\frac{u}{v^{1-\lambda}}}\Delta \log v+(1-\lambda) \bke{ \frac{u}{v^{1-\lambda}} }+\frac{\varphi }{v^{1-\lambda}} +\frac{2\theta}{v}\nabla \log v\cdot \nabla \psi. 
\end{equation}
Using Young's inequality and
\[
\nabla f\cdot \nabla \Delta f=\frac{1}{2}\Delta\abs{  \nabla f  }^{2}-\abs{ D^{2} f }^{2}  \,\mbox{ and }\,  \abs{\Delta f}\le \sqrt{d}\abs{D^{2} f} \,\mbox{ for }\, f\in\mathcal{C}^{2}(\overline{\Omega}),
\]
we compute the first  term on the right hand side of \eqref{UVEQ2}  as 
\begin{equation}\label{UVEQ3}
-\chi\bke{\frac{u}{v^{1-\lambda}}}\Delta \log v
\le 2\theta|D^{2}\log v|^{2}+(1-\lambda) \bke{\frac{u}{v^{1-\lambda}}}^{2}-\bigr{[}(1-\lambda)-\frac{d\chi^{2}}{8\lambda+4\chi} \bigr{]}\bke{\frac{u}{v^{1-\lambda}}}^{2}.
\end{equation}
Now, we consider three cases, \eqref{PR1_1}--\eqref{PR1_2}, separately. The definition of $\chi_{d,\lambda}$ gives
\begin{equation}\label{UVEQ4}
\begin{array}{ll}
(1-\lambda)-\displaystyle\frac{d\chi^{2}}{8\lambda+4\chi}=0\qquad \mbox{ if }\,  \chi= \chi_{d,\lambda},
\\
(1-\lambda)-\displaystyle\frac{d\chi^{2}}{8\lambda+4\chi}>0\qquad \mbox{ if }\,   \chi< \chi_{d,\lambda}.
\end{array}
\end{equation}

$\bullet$    $  (i) \, \, \chi\le \chi_{d,\lambda}$. \\
Using Young's inequality, we compute
\[
\frac{2\theta}{v}\nabla \log v\cdot \nabla \psi  \le (1-\lambda)\theta |\nabla \log v|^{2}+\frac{\theta}{  1-\lambda }|\nabla \psi |^{2}\frac{1}{v^2}.
\]
By \eqref{UVEQ2}--\eqref{UVEQ4}, Assumption~\ref{MODEL012}, and \eqref{LEM_LE1E}, it follows that
\begin{align*}
\partial_{t}&z-\Delta z-2\nabla z\cdot\nabla \log v 
\\
&\le (1-\lambda) \bke{ \frac{u}{v^{1-\lambda}} }+ (1-\lambda)\theta |\nabla \log v|^{2}+\frac{\varphi }{v^{1-\lambda}}+\frac{\theta}{ 1-\lambda }|\nabla \psi |^{2}\frac{1}{v^2}
\\
&
\le(1-\lambda) z+\frac{\|\varphi\|_{L^{\infty}(\Omega\times(0,\infty))}}{v^{1-\lambda}}+\frac{\theta}{ 1-\lambda }\|\nabla \psi\|_{L^{\infty}(\Omega\times(0,\infty))}^{2}\frac{1}{v^2}
\\
&\le(1-\lambda) z+\frac{\|\varphi\|_{L^{\infty}(\Omega\times(0,\infty))} e^{(1-\lambda)t}}{(\displaystyle\min_{\overline{\Omega}} v_{0} )^{1-\lambda}}+\frac{\theta}{ 1-\lambda }\frac{\|\nabla \psi\|_{L^{\infty}(\Omega\times(0,\infty))}^{2}e^{2t}}{(\displaystyle\min_{\overline{\Omega}} v_{0} )^{2}}
\\
&\le(1-\lambda) z+C_{1}e^{2t},
\end{align*}
where    $C_{1}$ is a positive constant independent of $t$.
Since
\[
\partial_{t}[ze^{-(1-\lambda) t}]-\Delta [ze^{-(1-\lambda) t}]-2\nabla [ze^{-(1-\lambda) t}]\cdot\nabla \log v
 \le
C_{1} e^{(1+\lambda)t},
\]
we have for $Z= ze^{-(1-\lambda) t}-\frac{C_{1}}{1+\lambda} e^{(1+\lambda)t}$ that
\[
\partial_{t}Z-\Delta Z-2\nabla Z\cdot\nabla \log v
 \le
0.
\]
As $\frac{\partial Z}{\partial\nu}\le0$  on $\partial\Omega$,  applying the maximum principle to $Z$-equation, we have  
\[
\| Z(\cdot,t)\|_{L^{\infty}(\Omega)}\le \|Z(\cdot,0)\|_{L^{\infty}(\Omega)}\qquad \mbox{ for } t<T_{\rm max}.
\]
Thus, \eqref{PR1_1} can be deduced.

$\bullet$    $ (ii) \, \, \chi< \chi_{d,\lambda}$, $\sigma>0$, and ${\displaystyle\inf_{s\ge0}}  \int_{\Omega}\psi(\cdot,s)>0$.   \\
From  \eqref{UVEQ2}--\eqref{UVEQ3}, we observe that
\begin{align}
\begin{aligned}\label{Prop1ii}
\partial_{t}&z-\Delta z-2\nabla z\cdot\nabla \log v 
\\
&+(1-\lambda)(\lambda+\chi)\bke{ \frac{u}{v^{1-\lambda}} }|\nabla \log v|^{2}+\bigr{[}(1-\lambda)-\frac{d\chi^{2}}{8\lambda+4\chi} \bigr{]}\bke{\frac{u}{v^{1-\lambda}}}^{2} 
\\
&+\sigma  v\bke{ \frac{u}{v^{1-\lambda}} }+(1-\lambda)  \bke{ \frac{u}{v^{1-\lambda}} } \frac{\psi }{v}+2\theta\frac{\psi }{v}|\nabla \log v|^{2}
\\
&\le 
 (1-\lambda) \bke{ \frac{u}{v^{1-\lambda}} }+\frac{\varphi }{v^{1-\lambda}}+\frac{2\theta}{v}\nabla \log v\cdot \nabla \psi,
 \end{aligned}
\end{align}
which yields
\begin{align*}
\partial_{t}&z-\Delta z-2\nabla z\cdot\nabla \log v 
\\
&+\bkt{    (1-\lambda)(\lambda+\chi)|\nabla \log v|^{2}    +\bigr{[}(1-\lambda)-\frac{d\chi^{2}}{8\lambda+4\chi} \bigr{]}\bke{\frac{u}{v^{1-\lambda}}}   +\sigma v - (1-\lambda)} \bke{ \frac{u}{v^{1-\lambda}} }
\\
&\le 
\frac{\varphi }{v^{1-\lambda}}+\frac{2\theta}{v}\nabla \log v\cdot \nabla \psi.
\end{align*}
Adding $\eqref{MODEL0}_{2}\times \frac{\sigma}{2}$ to the above inequality    and introducing $\mathcal{Z}:=z+\frac{\sigma}{2}v $, we have  
\begin{align*}
\partial_{t}&\mathcal{Z}-\Delta \mathcal{Z}-2\nabla \mathcal{Z}\cdot\nabla \log v 
\\
&+\bkt{    (1-\lambda)(\lambda+\chi)|\nabla \log v|^{2}    +\bigr{[}(1-\lambda)-\frac{d\chi^{2}}{8\lambda+4\chi} \bigr{]}\bke{\frac{u}{v^{1-\lambda}}}   +\frac{\sigma}{2} v - (1-\lambda)} \bke{ \frac{u}{v^{1-\lambda}} }
\\
&\le
 \frac{\varphi}{v^{1-\lambda}}+\frac{2\theta}{v}\nabla \log v\cdot \nabla \psi-\sigma v|\nabla \log v|^{2}-\frac{\sigma}{2}v+\frac{\sigma}{2}\psi.
\end{align*}
Using Young's inequality, we compute
\[
\frac{2\theta}{v}\nabla \log v\cdot \nabla \psi\le  \frac{\sigma}{2} v|\nabla \log v|^{2}+\frac{4\theta^{2}}{2\sigma v^3}|\nabla \psi|^{2}
\]
and using \eqref{UVEQ4}, we can find  a positive number $\varepsilon$ depending only on $\lambda$, $\chi$, and $d$ such that
 \[
\varepsilon \mathcal{Z}  \le       (1-\lambda)(\lambda+\chi)|\nabla \log v|^{2}    +\bigr{[}(1-\lambda)-\frac{d\chi^{2}}{8\lambda+4\chi} \bigr{]}\bke{\frac{u}{v^{1-\lambda}}}   +\frac{\sigma}{2} v.
 \]
It follows that
\[
\partial_{t}\mathcal{Z}-\Delta \mathcal{Z}-2\nabla \mathcal{Z}\cdot\nabla \log v+[\varepsilon \mathcal{Z}- (1-\lambda) ]\bke{ \frac{u}{v^{1-\lambda}} } 
\]\[
\le
 \frac{\varphi}{v^{1-\lambda}}+\frac{4\theta^{2}}{2\sigma v^3}|\nabla \psi|^{2}-\frac{\sigma}{2} v|\nabla \log v|^{2}-\frac{\sigma}{2}v+\frac{\sigma}{2}\psi.
\]Note that $v\ge \eta_{1}>0$ by $\eqref{LEM2_LE}$   and ${\displaystyle\inf_{s\ge0}}  \int_{\Omega}\psi(\cdot,s)>0$.
Using this lower bound and Assumption~\ref{MODEL012}, we   compute the right-hand side as   
\begin{align*}
& \frac{\varphi}{v^{1-\lambda}}+\frac{4\theta^{2}}{2\sigma v^3}|\nabla \psi|^{2}-\frac{\sigma}{2} v|\nabla \log v|^{2}-\frac{\sigma}{2}v+\frac{\sigma}{2}\psi
\\
&\le  \frac{\|\varphi\|_{L^{\infty}(\Omega\times(0,\infty))}}{\eta_{1}^{1-\lambda}}+\frac{4\theta^{2}}{2\sigma \eta_{1}^3}\|\nabla\psi\|_{L^{\infty}(\Omega\times(0,\infty))}^{2}-\frac{\sigma}{2} \eta_{1}|\nabla \log v|^{2}-\frac{\sigma}{2}v+\frac{\sigma}{2}\|\psi\|_{L^{\infty}(\Omega\times(0,\infty))}
\\
&=    C_{2}-\frac{\sigma}{2} \eta_{1}|\nabla \log v|^{2}-\frac{\sigma}{2}v, 
\end{align*}
where $C_{2} $ is a positive  constant  independent of $t$. We observe that
\begin{equation}\label{DET0}
\partial_{t}\mathcal{Z}-\Delta \mathcal{Z}-2\nabla \mathcal{Z}\cdot\nabla \log v
\le  C_{2}-\bket{[\varepsilon \mathcal{Z}- (1-\lambda) ]\bke{ \frac{u}{v^{1-\lambda}} } + \frac{\sigma}{2} \eta_{1}|\nabla \log v|^{2}+\frac{\sigma}{2}v }.
\end{equation}
Now, suppose that  there exists  $ x_{M} \in\overline{\Omega}$ such that
\begin{equation}\label{DET1}
\mathcal{Z}(x_M,t_M)=2\max\bket{ \| \mathcal{Z}_{0}\|_{L^{\infty}(\Omega)},\, \frac{2-\lambda}{\varepsilon},\, C_{2},\,   C_{2}\frac{2\theta}{ \sigma \eta_{1}}}
\end{equation}
for the first time $t_M>0$. Note that $t_M<T_{\rm max}$ by \eqref{LEM_LE2E}. Note also that the right-hand side of \eqref{DET0} is negative at $(x,t)=(x_M,t_M)$. Indeed,  
\[
C_{2} -\bket{ [\varepsilon \mathcal{Z}- (1-\lambda) ]\bke{ \frac{u}{v^{1-\lambda}} } + \frac{\sigma}{2} \eta_{1}|\nabla \log v|^{2}+\frac{\sigma}{2}v}\qquad\qquad\qquad\qquad\qquad\qquad\qquad\qquad  
\]\[
 \le C_{2}-\bket{ \bke{ \frac{u}{v^{1-\lambda}} } + \frac{\sigma}{2} \eta_{1}|\nabla \log v|^{2}+\frac{\sigma}{2}v}\qquad\qquad\qquad\qquad\qquad\qquad\qquad\qquad\qquad\qquad
\]\[
 \le  \begin{cases}
C_{2}-\displaystyle\bket{ \bke{ \frac{u}{v^{1-\lambda}} } + \theta|\nabla \log v|^{2}+\frac{\sigma}{2}v}=C_{2}-\mathcal{Z},\qquad\qquad\qquad\quad\,\,\,\,\,\,\, \mbox{ if }\,\, \theta\le \frac{\sigma}{2}\eta_{1},\\
C_{2}-\displaystyle\bket{ \frac{\frac{\sigma}{2} \eta_{1}}{\theta}\bke{ \frac{u}{v^{1-\lambda}} } + \frac{\sigma}{2} \eta_{1}|\nabla \log v|^{2}+\frac{\frac{\sigma}{2} \eta_{1}}{\theta}\frac{\sigma}{2}v}=C_{2}-\frac{\frac{\sigma}{2} \eta_{1}}{\theta}\mathcal{Z},\quad \mbox{ if }\,\, \theta> \frac{\sigma}{2}\eta_{1},
\end{cases}
\]
which is strictly negative for $(x,t)=(x_M,t_M)$.
If $ x_{M}$ is  an interior point of  $  \Omega  $, then    
\[
\partial_{t}\mathcal{Z}-\Delta \mathcal{Z}-2\nabla \mathcal{Z}\cdot\nabla \log v\ge0\qquad\mbox{ for } (x,t)=(x_{M},t_{M})
\]
and thus, this with  \eqref{DET0}  
leads to the contradiction.
Let $ x_{M}\in\partial \Omega$. Then, by Hopf's lemma type argument,  
$\frac{\partial \mathcal{Z}}{\partial\nu}$ is strictly positive at $(x,t)=(x_M,t_M)$ but  this leads to the contradiction because   $\frac{\partial \mathcal{Z}}{\partial\nu}\le 0$ on $\partial\Omega$ for $t<T_{\rm max}$.
Therefore, there is no such  $(x_M,t_{M})\in \overline{\Omega}\times(0,T_{\rm max})$ satisfying \eqref{DET1}. Thus,
\[
\|\mathcal{Z}(\cdot,t)\|_{L^{\infty}(\Omega)}< 2\max\bket{ \| \mathcal{Z}_{0}\|_{L^{\infty}(\Omega)},\, \frac{2-\lambda}{\varepsilon},\, C_{2},\,   C_{2}\frac{2\theta}{ \sigma \eta_{1}}},\qquad  t<T_{\rm max},
\]
namely,   \eqref{PR1_3} is obtained.
 
$\bullet$    $(iii) \, \, \chi< \chi_{d,\lambda}$ and   $(\varphi(x,s),\psi(x,s))=(0,b(s))$ for all $x\in \Omega$, $s>0$. \\
As in \eqref{Prop1ii},  by \eqref{UVEQ2}--\eqref{UVEQ3}, we have
\begin{align*}
\partial_{t}&z-\Delta z-2\nabla z\cdot\nabla \log v 
\\
&+(1-\lambda)(\lambda+\chi)\bke{ \frac{u}{v^{1-\lambda}} }|\nabla \log v|^{2}+\bigr{[}(1-\lambda)-\frac{d\chi^{2}}{8\lambda+4\chi} \bigr{]}\bke{\frac{u}{v^{1-\lambda}}}^{2} 
\\
&+\sigma  v\bke{ \frac{u}{v^{1-\lambda}} }+(1-\lambda)  \bke{ \frac{u}{v^{1-\lambda}} } \frac{\psi }{v}+2\theta\frac{\psi }{v}|\nabla \log v|^{2}
\\
&\le 
 (1-\lambda) \bke{ \frac{u}{v^{1-\lambda}} }
\end{align*}
and thus,
\begin{align*}
\partial_{t}&z-\Delta z-2\nabla z\cdot\nabla \log v 
\\
&+\bkt{(1-\lambda)(\lambda+\chi) |\nabla \log v|^{2}+\bigr{[}(1-\lambda)-\frac{d\chi^{2}}{8\lambda+4\chi} \bigr{]}\bke{\frac{u}{v^{1-\lambda}}}  - (1-\lambda) }\bke{\frac{u}{v^{1-\lambda}}}
\le 
0.
\end{align*}
Note from \eqref{UVEQ4} that
 there exists a positive number $\varepsilon$ depending only on $\lambda$, $\chi$, and $d$   such that
\[
\partial_{t}z-\Delta z-2\nabla z\cdot\nabla \log v +[\varepsilon z-(1-\lambda)  ]\bke{\frac{u}{v^{1-\lambda}}}
\le 0.
\]
By the maximum principle, we have
\[
\|z(\cdot,t)\|_{L^{\infty}(\Omega)}\le \max \bket{ \|z_0\|_{L^{\infty}(\Omega)},\, \frac{ 1-\lambda}{\varepsilon} },\qquad   t<T_{\rm max}.
\]
Therefore, \eqref{PR1_2} is obtained.
This  completes the proof.
\end{proof}

To obtain Theorem~\ref{THM1},  we prepare two bounds of $u$ in the following two lemmas.
First, we estimate the temporal bound of  $u$ for $\chi\le \chi_{d,\lambda}$. 
\begin{lemma}\label{LEMTHM1-1}
Let the same assumptions as in Proposition~\ref{PROP1} be satisfied. Suppose that $\chi\le \chi_{d,\lambda}$ and $T<T_{\rm max} $. Then, there exists a positive number $C$ independent of  $T$  such that
\begin{equation}\label{LEMTHM1-11}
\sup_{ t\le T}\| u(\cdot,t) \|_{L^{\infty}(\Omega)}\le Ce^{CT}.
\end{equation}
\end{lemma}
\begin{proof}
Integrating the first equation of \eqref{MODEL0} over $\Omega$ and using a direct computation,  we have  
\[
\frac{d}{dt}\int_{\Omega}u=-\sigma\int_{\Omega}uv +\int_{\Omega}\varphi\le \int_{\Omega}\varphi\le |\Omega|\|\varphi\|_{L^{\infty}(\Omega\times(0,\infty))},\qquad t<T_{\rm max},
\]
which entails by integrating  with respect to time that
\begin{equation}\label{THM1_pf1}
\|u(\cdot,t)\|_{L^{1}(\Omega)}\le \|u_0\|_{L^{1}(\Omega)}+ |\Omega|\|\varphi\|_{L^{\infty}(\Omega\times(0,\infty))}t,\qquad t<T_{\rm max}.
\end{equation}
Using the representation formula of $u$ and 
 $- \sigma u v\le0$, we note that
\[
u(t)\le e^{t\Delta}u_0+\int_{0}^{t} e^{(t-s) \Delta }[ -\chi\nabla \cdot \bke{ u  \nabla \log v    } +\varphi](s)ds,\qquad t<T_{\rm max}.
\]
By the smoothing estimate for Neumann heat semigroup $e^{t\Delta}$ and   
\begin{equation}\label{MPINF}
\|e^{t\Delta}f\|_{L^{\infty}(\Omega)}\le \| f\|_{L^{\infty}(\Omega)}\qquad \mbox{ for } f\in L^{\infty}(\Omega),
\end{equation} we have
\[
\|u(t)\|_{L^{\infty}(\Omega)}
\le  \|u_0\|_{L^{\infty}(\Omega)}+C\int_{0}^{t}  (1+ (t-s)^{-\frac{1}{2}-\frac{d}{2}\frac{1}{2d}})\|   u  \nabla \log v (s)  \|_{L^{2d}(\Omega)}
+  \|   \varphi (s)  \|_{L^{\infty}(\Omega)}ds.
\]
Using H\"older's inequality, \eqref{PR1_1},   \eqref{THM1_pf1}, and a direct computation,  we further compute the  integral  on the right-hand side as
\[
 \int_{0}^{t}  ( 1+(t-s)^{-\frac{1}{2}-\frac{d}{2}\frac{1}{2d}})\|   u  \nabla \log v (s)  \|_{L^{2d}(\Omega)}+  \|   \varphi (s)  \|_{L^{\infty}(\Omega)}ds
\]
\[
 \le   \int_{0}^{t}  (1+(t-s)^{-\frac{3}{4} })\|   u (s)  \|_{L^{1}(\Omega)}^{\frac{1}{2d}}\|   u (s)  \|_{L^{\infty}(\Omega)}^{1-\frac{1}{2d}}\|    \nabla \log v (s)  \|_{L^{\infty}(\Omega)}ds+\|\varphi\|_{L^{\infty}(\Omega\times(0,\infty))}t
\]
 \[
 \le  C \int_{0}^{t}  (1+ (t-s)^{-\frac{3}{4} })(\|u_0\|_{L^{1}(\Omega)}+ |\Omega|\|\varphi\|_{L^{\infty}(\Omega\times(0,\infty))}s)^{\frac{1}{2d}}  e^{s}ds\sup_{  s\le t}\| u(s)\|_{L^{\infty}(\Omega)}^{1-\frac{1}{2d}}+Ct
\]
\[
\le C(1+t^{1+\frac{1}{2d}})e^{t}\sup_{ s\le t}\| u(s)\|_{L^{\infty}(\Omega)}^{1-\frac{1}{2d}}+Ct,
\]
\[
\le Ce^{2t}\sup_{ s\le t}\| u(s)\|_{L^{\infty}(\Omega)}^{1-\frac{1}{2d}}+Ct,
\]
where $C$ is a positive constant independent of $t$. Combining above estimates and taking supremum over $0\le t\le T$, we have  
\[
\sup_{t\le T}\|u( t)\|_{L^{\infty}(\Omega)}\le  Ce^{2T}\sup_{ t\le T}\| u( t)\|_{L^{\infty}(\Omega)}^{1-\frac{1}{2d}}+C(1+T).
\]
This yields by Young's inequality and  $1+T\le e^{T}$
that  \eqref{LEMTHM1-11}. This completes the proof.
\end{proof}

Next,  we obtain the uniform bound of  $u$ for $\chi< \chi_{d,\lambda}$ when \eqref{THM1ASS1} holds.
\begin{lemma}\label{LEMTHM1-2}
Let the same assumptions as in  Proposition~\ref{PROP1} be satisfied. Assume that $\chi< \chi_{d,\lambda}$ and  \eqref{THM1ASS1} holds. Then,   there exists a positive constant $C$    independent of $T<T_{\rm max}$ satisfying 
\[
\sup_{ t\le T}\| u(\cdot,t) \|_{L^{\infty}(\Omega)}\le C.
\]
\end{lemma}
\begin{proof}
Let $\chi< \chi_{d,\lambda}$. In   the case where  $\sigma>0$ and  ${\displaystyle\inf_{t\ge0}}  \int_{\Omega} \psi(\cdot,t) >0$,   
it   follows by  \eqref{PR1_3} in Proposition~\ref{PROP1} that 
\[
 \norm{u}_{L^{\infty}((0,T_{\rm max}); L^{\infty}(\Omega))}
  \le   \norm{\frac{u}{v^{1-\lambda}} }_{L^{\infty}((0,T_{\rm max}); L^{\infty}(\Omega))} \norm{v}_{L^{\infty}((0,T_{\rm max}); L^{\infty}(\Omega))}^{1-\lambda}\le C.
  \]
Next, in the case of
 $(\varphi(x,t),\psi(x,t))=(0,b(t))$ for all $x\in \Omega$, $t>0$, we note that
 \[
 \partial_{t}u+u-\Delta u=-\chi\nabla \cdot(u\nabla\log v)-\sigma uv+\varphi+u\le -\chi\nabla \cdot(u\nabla\log v)+u
 \]
and thus, 
\[
u(t)\le e^{t(\Delta-1)}u_0+\int_{0}^{t} e^{(t-s) (\Delta-1) }[ -\chi\nabla \cdot \bke{ u  \nabla \log v    } +u ](s)ds,\qquad t<T_{\rm max}.
\]
By  the smoothing estimate for $e^{t\Delta}$, 
 \eqref{PR1_2},  \eqref{THM1_pf1} with $\varphi=0$, \eqref{MPINF}, and H\"older's inequality we obtain 
 \begin{align*}
 \|u(t)\|_{L^{\infty}(\Omega)}
 &\le  \|u_0\|_{L^{\infty}(\Omega)}+C\int_{0}^{t} e^{-(t-s)}(1+ (t-s)^{-\frac{1}{2}-\frac{d}{2}\frac{1}{2d}})\|   u  \nabla \log v (s)  \|_{L^{2d}(\Omega)}ds
\\
&\quad+ C\int_{0}^{t} e^{-(t-s)}  (1+(t-s)^{-\frac{d}{2}\frac{1}{2d}})\|   u (s)  \|_{L^{2d}(\Omega)}   ds
\\
&\le C+C\int_{0}^{t} e^{-(t-s)} (1+(t-s)^{-\frac{3}{4} })\|   u (s)  \|_{L^{1}(\Omega)}^{\frac{1}{2d}}\|   u (s)   \|_{L^{\infty}(\Omega)}^{1-\frac{1}{2d}}\|      \nabla \log v (s)  \|_{L^{\infty}(\Omega)}ds
\\
&\quad+ C\int_{0}^{t} e^{-(t-s)}   (1+(t-s)^{-\frac{1}{4} })\|   u (s)  \|_{L^{1}(\Omega)}^{\frac{1}{2d}}\|   u (s)   \|_{L^{\infty}(\Omega)}^{1-\frac{1}{2d}}   ds 
\\
&\le  C+C\int_{0}^{t} e^{-(t-s)} (1+(t-s)^{-\frac{3}{4} }+ (t-s)^{-\frac{1}{4} }) ds\sup_{  s\le t}\| u(s)\|_{L^{\infty}(\Omega)}^{1-\frac{1}{2d}}  
\\
&\le C+C \sup_{  s\le t}\| u(s)\|_{L^{\infty}(\Omega)}^{1-\frac{1}{2d}},\qquad t<T_{\rm max}.
\end{align*}
Taking supremum over $0\le t\le T$ and using  Young's inequality,  we have the desired uniform bound. This completes the proof.
\end{proof}

We are ready to prove Theorem~\ref{THM1}.
\begin{pfthm1}
We first obtain $T_{\rm max}=\infty$ for   $\chi\le \chi_{d,\lambda}$.  Suppose not, i.e.  $T_{\rm max}<\infty$.
We show that  for   $q>d$, there exists a positive constant $C$ independent of  $T<T_{\rm max}$ such that
 \begin{equation}\label{pfthm1_1}
\sup_{t\le T}\| v(\cdot,t) \|_{W^{1,q}(\Omega)}\le Ce^{CT}.
\end{equation}
Let $t\le T<T_{\rm max}$.  We begin with recalling   the representation formula  
 \begin{equation}\label{VREP}
 v(t)=e^{t(\Delta-1)}v_0+\int_{0}^{t}e^{(t-s)(\Delta-1)}[uv^{\lambda}+\psi](s)ds.
 \end{equation}
Using the smoothing estimate for $e^{t\Delta}$ and H\"older's inequality, we   compute
\begin{align*}
&\|v(t)\|_{W^{1,q}(\Omega)}
\\
&\le C\|v_0\|_{W^{1,q}(\Omega)}+C\int_{0}^{t}e^{-(t-s)}(1+(t-s)^{-\frac{1}{2}})[\|uv^{\lambda}(s)\|_{L^{\infty}(\Omega)}+\|\psi(s)\|_{L^{\infty}(\Omega)}]ds
\\
&\le C+C\int_{0}^{t}e^{-(t-s)}(1+(t-s)^{-\frac{1}{2}}) [\|u(s)\|_{L^{\infty}(\Omega)}\|v(s)\|_{L^{\infty}(\Omega)}^{\lambda}+  \|\psi\|_{L^{\infty}(\Omega\times(0,\infty))}]ds.
\end{align*}
Using the embedding relation $W^{1,q}(\Omega)\hookrightarrow L^{\infty}(\Omega)$, \eqref{LEMTHM1-11}, and a direct computation, 
we further compute the integral on the right hand side as
\begin{align*}
\int_{0}^{t}&e^{-(t-s)}(1+(t-s)^{-\frac{1}{2}}) [\|u(s)\|_{L^{\infty}(\Omega)}\|v(s)\|_{L^{\infty}(\Omega)}^{\lambda}+  \|\psi\|_{L^{\infty}(\Omega\times(0,\infty))}]ds
\\
&\le 
 \int_{0}^{t}e^{-(t-s)}(1+(t-s)^{-\frac{1}{2}})ds\sup_{  s \le t}  \|u(s)\|_{L^{\infty}(\Omega)} \sup_{  s \le t}  \|v(s)\|_{W^{1,q}(\Omega)}^{\lambda}+C
 \\
& \le Ce^{Ct} \sup_{  s \le t}  \|v(s)\|_{W^{1,q}(\Omega)}^{\lambda}+C,
\end{align*}
where   $C$ is a positive constant independent of $t$. Combining above estimates, taking supremum over $0\le t\le T$,  and using Young's inequality, we have \eqref{pfthm1_1}.   This leads to the contradiction to the fact, due to     the blow-up criterion \eqref{LEM_LE2E} and Lemma~\ref{LEMTHM1-1}, that  $T_{\rm max}<\infty$. Thus,   $T_{\rm max}=\infty$.

Next, we obtain \eqref{THM1_2} under $\chi< \chi_{d,\lambda}$ and \eqref{THM1ASS1}. Thanks to Lemma~\ref{LEMTHM1-2} and \eqref{PR1_3}--\eqref{PR1_2},  it is sufficient to show that $v$ has a uniform  bound.  When  $\sigma>0$ and  ${\displaystyle\inf_{t\ge0} } \int_{\Omega}\psi(\cdot,t) >0$ are satisfied, the uniform bound of $v$ is  obtained in  \eqref{PR1_3}.   In the case where  
 $(\varphi(x,t),\psi(x,t))=(0,b(t))$ for all $x\in \Omega$, $t>0$, using \eqref{MPINF}, \eqref{VREP}, and Lemma~\ref{LEMTHM1-2}, we compute
\[
\|v(t)\|_{L^{\infty}(\Omega)}
\le  \|v_0\|_{L^{\infty}(\Omega)}+C\int_{0}^{t}e^{-(t-s)} [\|uv^{\lambda}(s)\|_{L^{\infty}(\Omega)}+\|\psi(s)\|_{L^{\infty}(\Omega)}]ds
\]
\[
\le C+C\int_{0}^{t}e^{-(t-s)}  [\|u(s)\|_{L^{\infty}(\Omega)}\|v(s)\|_{L^{\infty}(\Omega)}^{\lambda}+  \|\psi\|_{L^{\infty}(\Omega\times(0,\infty))}]ds 
\]
\[
\le C+C\int_{0}^{t}e^{-(t-s)} ds   \sup_{  s \le t}\|v(s)\|_{L^{\infty}(\Omega)}^{\lambda}
\le   C+C  \sup_{  s \le t}\|v(s)\|_{L^{\infty}(\Omega)}^{\lambda},
\]
where $C$ is a positive constant independent of $t$. Then, taking supremum over time interval and using Young's inequality, we can conclude that $v$ is uniformly bounded. This completes the proof.
\end{pfthm1}

\section{Global well-posedness for small data}\label{SEC4}
In this section, we prove  Theorem~\ref{THM2} and Theorem~\ref{THM2.5}. We prepare the lower and upper estimates for $v$, when $ u \le \frac{1}{2}$ and $\displaystyle\min_{\overline{\Omega}} v_0  \ge \displaystyle\inf_{x\in\overline{\Omega},\\s\ge0} \psi(x,s) >0$.
  $v$ is bounded above and bounded below away from zero.
\begin{lemma}\label{VLEM}
Let $\Omega$ be a smooth,  bounded and convex domain of $\R^{d}$, $d\ge2$.   Suppose that  $(u_0,v_0,\varphi,\psi)$ satisfies Assumption~\ref{MODEL012}, and   $\sigma\ge0$, $0\le \lambda\le1$. Let $(u,v)$ be the solution    to \eqref{MODEL0}  given by Lemma~\ref{LEM_LE}.   If 
\[
\|u\|_{L^{\infty}(0,T_{\rm max}; L^{\infty}(\Omega))}\le\frac{1}{2}
\]
 and 
\[
\displaystyle\min_{\overline{\Omega}} v_0  \ge \displaystyle\inf_{x\in\overline{\Omega},\\s\ge0} \psi(x,s) >0,
\]
then  
\[
\displaystyle\inf_{x\in\overline{\Omega},\\s\ge0} \psi(x,s) \le v(x,t)\le 2\max\{   \|v_0\|_{L^{\infty}(\Omega)},\,\|\psi\|_{L^{\infty}(\Omega\times(0,\infty))}\},\qquad  x\in\Omega, \,\,t<T_{\rm max}. 
\]
\end{lemma}
\begin{proof}
Let $t < T_{\rm max}$. Using the representation formula for $\partial_{t}v+(\frac{1}{2}-\Delta) v=uv-\frac{1}{2}v +\psi$, and $
\|u\|_{L^{\infty}(0,T_{\rm max}; L^{\infty}(\Omega))}\le\frac{1}{2}$, we compute
\[
v(t)=e^{t(\Delta-\frac{1}{2})}v_0+\int_{0}^{t} e^{(t-s)(\Delta-\frac{1}{2})}(uv-\frac{1}{2}v+\psi)(s)ds
\]
\[
\le e^{t(\Delta-\frac{1}{2})}v_0+\int_{0}^{t} e^{(t-s)(\Delta-\frac{1}{2})}\psi(s)ds.
\]
Using \eqref{MPINF}, it follows that
\[
\norm{v(t)}_{L^{\infty}}\le \norm{v_0}_{L^{\infty}}e^{-\frac{1}{2}t}+2(1-e^{-\frac{1}{2}t})\|\psi \|_{L^{\infty}(\Omega\times(0,\infty))  }
\]
and thus,  a desired upper bound for $v$ is obtained. The desired lower bound for $v$ is a consequence of \eqref{LEM2_LE2} in Lemma~\ref{UNIFLOWV}. This completes the proof.
\end{proof}

We are ready to prove Theorem~\ref{THM2} and Theorem~\ref{THM2.5}.
Again, we use the  maximum principle.
\begin{pfthm2}     
Let  $t<T_{\rm max}$ and  $z =u+\theta|\nabla \log v|^{2}$, where   $\theta =1+\frac{\chi}{2}$.  We recall from \eqref{ZBCCON}--\eqref{UVEQ2} that $\frac{\partial z}{\partial\nu}\le0$ on $\partial\Omega$, and
\[
\partial_{t}z-\Delta z-2\nabla z\cdot\nabla \log v 
 + \sigma vu+2\theta|D^{2}\log v|^{2}+2\theta\frac{\psi }{v}|\nabla \log v|^{2}
\]
\[
=
-\chi u \Delta \log v + \varphi  +\frac{2\theta}{v}\nabla \log v\cdot \nabla \psi.
\] Using the pointwise estimate 
$
\abs{\Delta f}\le \sqrt{d}\abs{D^{2} f}$ for $f\in\mathcal{C}^{2}(\overline{\Omega})$   and Young's inequality, we can compute
\begin{equation}\label{THM2_1}
-\chi u\Delta \log v\le 2\theta|D^{2}\log v|^{2}+\frac{d\chi^{2}}{8\theta}u^{2},
\end{equation}
and
\begin{equation}\label{THM2_2}
\frac{2\theta}{v}\nabla \log v\cdot \nabla \psi 
\le \theta\frac{\psi }{v}|\nabla \log v|^{2}+ \theta \frac{ |\nabla \psi |^{2} }{ \psi v  }.
\end{equation}
Thus, we have
\[
\partial_{t}z-\Delta z-2\nabla z\cdot\nabla \log v 
 +\sigma vu+\theta\frac{\psi }{v}|\nabla \log v|^{2}
\le \frac{d\chi^{2}}{8\theta}u^{2}+\varphi + \theta \frac{ |\nabla \psi |^{2} }{ \psi v  },
\]
and which entails by $v\ge \eta$, $\psi\ge \eta$, and $u\le z$ that
\[
\partial_{t}z-\Delta z-2\nabla z\cdot\nabla \log v +\eta \bke{ \sigma  -\frac{d\chi^{2}}{8\theta\eta }z }u
+\frac{\eta }{v}\theta|\nabla \log v|^{2}
\le   \varphi + \theta \frac{ |\nabla \psi |^{2} }{ \psi \eta   }.
\]
Note that  $z_0<\delta=\min\bket{\frac{1}{2},  \frac{ 4\sigma\theta\eta }{ d\chi^{2}} }$ by \eqref{ASSTHM21}. Now, 
suppose that there exists $x_{\delta} \in\overline{\Omega} $ such that
$z(x_{\delta},t_{\delta})=\delta$ for the first time $t_{\delta}<T_{\rm max}$.
 Using the definition of $\delta$ and Lemma~\ref{VLEM}, we have for $(x,t)=(x_{\delta},t_{\delta})$ that
\begin{align*}
\eta & \bke{  \sigma-\frac{d\chi^{2}}{8\theta \eta }z   }u
+\frac{ \eta }{v}\theta|\nabla \log v|^{2}
\\
&\ge \frac{\eta }{2}\sigma u
+\frac{ \eta }{2\max\{   \|v_0\|_{L^{\infty}(\Omega)},\,\|\psi\|_{L^{\infty}(\Omega\times(0,\infty))}\}}\theta|\nabla \log v|^{2}
\\
&\ge \frac{\eta }{2}\min\bket{ \sigma,\,\frac{1}{ \max\{   \|v_0\|_{L^{\infty}(\Omega)},\,\|\psi\|_{L^{\infty}(\Omega\times(0,\infty))}\}} }z
\\
&=\frac{\eta }{2}\frac{1}{\max\{  \sigma^{-1},\,    \|v_0\|_{L^{\infty}(\Omega)},\,\|\psi\|_{L^{\infty}(\Omega\times(0,\infty))}   \}}\delta
\end{align*}
which yields for $(x,t)=(x_{\delta},t_{\delta})$ that
\begin{align}
\begin{aligned}\label{THM2PF}
\partial_{t}&z-\Delta z-2\nabla z\cdot\nabla \log v 
\\
&\le 
  \varphi  +\frac{4\theta}{\eta } |\nabla \sqrt{\psi}|^{2} -\frac{\eta }{2}\frac{\delta}{\max\{ \sigma^{-1},\,    \|v_0\|_{L^{\infty}(\Omega)},\,\|\psi\|_{L^{\infty}(\Omega\times(0,\infty))}   \}}.
\end{aligned}
\end{align}
Note that by  \eqref{ASSTHM22}, the right-hand side is strictly negative. If $x_{\delta}$ is   an interior point of $\Omega$, then  the left-hand side of \eqref{THM2PF} is nonnegative.
Thus, $x_{\delta}$ is not an interior point of $\Omega$.  Let $x_{\delta}\in \partial \Omega$. Then, by Hopf's lemma type argument,  $\frac{\partial z}{\partial\nu}>0$ at $(x,t)=(x_{\delta},t_{\delta})$ but again, this leads to the contradiction since
 $\frac{\partial z}{\partial\nu}\le0$ on $\partial\Omega$.
 Therefore, $z<\delta$ for $t<T_{\rm max}$. Since  $u\le z$,  $v$ has a uniform bound    by Lemma~\ref{VLEM}, and $\nabla v=v\nabla \log v\le vz^{\frac{1}{2}}$   also has a uniform  bound. Then,  by \eqref{LEM_LE2E} and $\|\nabla v\|_{L^{q}(\Omega)}\le C\|\nabla v\|_{L^{\infty}(\Omega)}$,  we obtain $T_{\rm max}=\infty$ and \eqref{THM1_2}. This completes the proof.   
\end{pfthm2}
\begin{pfthm3} 
The proof is similar to the proof of Theorem~\ref{THM2}.  Let  $t<T_{\rm max}$ and  $z =\frac{u}{v^{1-\lambda}}+\theta|\nabla \log v|^{2}$, where   $\theta=\lambda+\frac{\chi}{2}$.  
Again, we recall from \eqref{ZBCCON}--\eqref{UVEQ2} that $\frac{\partial z}{\partial\nu}\le0$ on $\partial\Omega$, and
\[
\partial_{t}z-\Delta z-2\nabla z\cdot\nabla \log v 
+(1-\lambda)(\lambda+\chi)\bke{ \frac{u}{v^{1-\lambda}} }|\nabla \log v|^{2}+(1-\lambda) \bke{\frac{u}{v^{1-\lambda}}}^{2}
\]
 \[
 +\sigma  v\bke{ \frac{u}{v^{1-\lambda}} }+(1-\lambda)  \bke{ \frac{u}{v^{1-\lambda}} } \frac{\psi }{v}+2\theta|D^{2}\log v|^{2}+2\theta\frac{\psi }{v}|\nabla \log v|^{2}
 \]
\[
=
-\chi\bke{\frac{u}{v^{1-\lambda}}}\Delta \log v+(1-\lambda) \bke{ \frac{u}{v^{1-\lambda}} }+\frac{\varphi }{v^{1-\lambda}} +\frac{2\theta}{v}\nabla \log v\cdot \nabla \psi. 
\]
By \eqref{THM2_1}--\eqref{THM2_2}, we have
\begin{align*}
\partial_{t}&z-\Delta z-2\nabla z\cdot\nabla \log v 
+(1-\lambda)(\lambda+\chi)\bke{ \frac{u}{v^{1-\lambda}} }|\nabla \log v|^{2} 
\\
& +\sigma  v\bke{ \frac{u}{v^{1-\lambda}} }+(1-\lambda)  \bke{ \frac{u}{v^{1-\lambda}} } \frac{\psi }{v} +\theta\frac{\psi }{v}|\nabla \log v|^{2}
\\
&\le 
\bke{    \frac{d\chi^2}{8\theta} -(1-\lambda)        }\bke{ \frac{u}{v^{1-\lambda}} }^2
+(1-\lambda) \bke{ \frac{u}{v^{1-\lambda}} }+\frac{\varphi }{v^{1-\lambda}} + \theta \frac{  |\nabla \psi|^2  }{\psi v}, 
\end{align*}
  which entails by $v\ge \eta$, $\psi\ge \eta$ and $\frac{u}{v^{1-\lambda}}\le z$ that
\begin{align*}
\partial_{t}&z-\Delta z-2\nabla z\cdot\nabla \log v 
\\
& +\bkt{\sigma  \eta-(1-\lambda) -\bke{    \frac{d\chi^2}{8\theta} -(1-\lambda)        }z   }\bke{ \frac{u}{v^{1-\lambda}} }+ \theta\frac{\eta }{v}|\nabla \log v|^{2}
\\
&\le 
 \frac{\varphi }{\eta^{1-\lambda}} + \frac{4\theta}{\eta } |\nabla \sqrt{\psi}|^{2}.
\end{align*}
Now, note from \eqref{THM3_2} that  
\[
z_0<\delta_{0}= \min\bket{\frac{1}{2},\,\,  \frac{\sigma \eta}{2}\bke{  \frac{d\chi^2}{8\theta}-(1-\lambda)    }^{-1} }.
\]   As in the proof of Theorem~\ref{THM2}, we show   $z<\delta_{0}$ for $t<T_{\rm max}$. 
Suppose that there exists $(x_{\delta},t_{\delta})\in\overline{\Omega}\times(0,T_{\rm max})$ such that
$z(x_{\delta},t_{\delta})=\delta_{0}$ for the first time. Using the definition of $\delta_{0}$, $\sigma\eta\ge 4(1-\lambda)$,  and Lemma~\ref{VLEM}, we have for $(x,t)=(x_{\delta},t_{\delta})$ that 
\begin{align*}
 &\bkt{\sigma  \eta-(1-\lambda) -\bke{    \frac{d\chi^2}{8\theta} -(1-\lambda)        }z   }\bke{ \frac{u}{v^{1-\lambda}} }+ \theta\frac{\eta }{v}|\nabla \log v|^{2}
\\
& \ge  \bkt{\sigma  \eta-(1-\lambda) -\frac{\sigma \eta}{2}   }\bke{ \frac{u}{v^{1-\lambda}} }+ \theta\frac{\eta }{v}|\nabla \log v|^{2}
\\
& \ge \frac{\eta\sigma}{ 4} \bke{ \frac{u}{v^{1-\lambda}} }+\frac{ \eta }{2\max\{   \|v_0\|_{L^{\infty}(\Omega)},\,\|\psi\|_{L^{\infty}(\Omega\times(0,\infty))}\}}\theta|\nabla \log v|^{2}
\\
& \ge \frac{\eta }{2}\min\bket{ \frac{\sigma}{2},\,\frac{1}{ \max\{   \|v_0\|_{L^{\infty}(\Omega)},\,\|\psi\|_{L^{\infty}(\Omega\times(0,\infty))}\}} }z
\\
& =\frac{\eta }{2}\frac{1}{\max\{ 2 \sigma^{-1},\,    \|v_0\|_{L^{\infty}(\Omega)},\,\|\psi\|_{L^{\infty}(\Omega\times(0,\infty))}   \}}\delta_{0},
\end{align*}
which implies for $(x,t)=(x_{\delta},t_{\delta})$ that
\begin{align*}
\partial_{t}&z-\Delta z-2\nabla z\cdot\nabla \log v
\\
&\le 
 \frac{\varphi }{\eta^{1-\lambda}} + \frac{4\theta}{\eta } |\nabla \sqrt{\psi}|^{2} -\frac{\eta }{2}\frac{\delta_{0}}{\max\{ 2\sigma^{-1},\,    \|v_0\|_{L^{\infty}(\Omega)},\,\|\psi\|_{L^{\infty}(\Omega\times(0,\infty))}   \}}.
\end{align*}
 Then, the right-hand side is negative by \eqref{THM3_3}.  
  The remaining part of the proof is very similar to that of Theorem~\ref{THM2}, and thus the details are omitted.
 This completes the proof.
\end{pfthm3}

 \section{ Long-time asymptotics }\label{SEC5}
 In this section, we prove  the long time behavior result, Theorem~\ref{THM3}.  To this end, first, we estimate  fractional norms  of $u$ and $v$ and using them, we 
 compute $L^{\infty}$-norms of  $u$ and $v$.
  Below, $A$
denotes the sectorial realization of $-\Delta+1$ in $L^{r}(\Omega)$ with $1<r<\infty$  under homogeneous Neumann boundary condition, and  $A^{\beta}$ with $\beta\in(0,1)$ denotes the fractional power of $A$(see, e.g. \cite[Section 1.4]{H81}).  
\begin{lemma}\label{LEM51}
Let $\Omega$ be a smooth, bounded and convex domain of $\R^{d}$, $d\ge2$, and   $\sigma>0$, $0\le \lambda\le 1$. Suppose that $(u, v)$ is a unique  global classical solution to  \eqref{MODEL0}
satisfying \eqref{THM1_2}. Then, for $r>d$ and $\frac{d}{2r}<\beta<\frac{1}{2}$, there exist a positive constant   $C$ independent of $t$ such that
\begin{equation}\label{LEM101}
\|A^{\beta}u(\cdot,t)\|_{L^{r}( \Omega )}+\|A^{\beta}v(\cdot,t)\|_{L^{r}( \Omega )}\le C\qquad \mbox{ for all }  t>0.
\end{equation}
Moreover, there exist  $\gamma\in(0,1)$ and a positive constant   $C$ independent of $t$ such that
\begin{equation}\label{LEM102}
\|u(\cdot,t)\|_{L^{\infty}( \Omega )}\le C\|u(\cdot,t)\|_{L^{1}( \Omega )}^{\frac{\gamma}{\gamma+r(1-\gamma)}}\qquad \mbox{ for all }  t>0,
\end{equation}
\begin{equation}\label{LEM103}
\|v(\cdot,t)-v_{\infty}\|_{L^{\infty}( \Omega )}\le C\|v(\cdot,t)-v_{\infty}\|_{L^{2}( \Omega )}^{\frac{2\gamma}{2\gamma+r(1-\gamma)}}\qquad \mbox{ for all }  t>0,
\end{equation}
where $ v_{\infty}$ denotes the solution for \eqref{VINFMODEL}.
\end{lemma}
 \begin{proof}
First, we show \eqref{LEM101}.
 Let $r>d\ge2$ and $\frac{d}{2r}<\beta<\frac{1}{2}$.
Using the representation  formula of 
$
u$,
we compute
\begin{align*}
&\|A^{\beta}u(t)\|_{L^{r}(\Omega)}
\\
&
\le \|A^{\beta}e^{t(\Delta-1)}u_0\|_{L^{r}(\Omega)}+\int_{0}^{t}\|A^{\beta}e^{(t-s)(\Delta-1)}\bke{-\chi\nabla \cdot \bke{ u  \nabla \log v    }+u- \sigma u v +\varphi}(s)\|_{L^{r}(\Omega)}ds
\\
&\le C\|u_{0}\|_{W^{1,r}(\Omega)}+C\int_{0}^{t}e^{-(t-s)}\bke{1+(t-s)^{-\beta-\frac{1}{2}}}\|u  \nabla \log v(s)\|_{L^{r}(\Omega)} ds
\\
&\quad+C\int_{0}^{t}e^{-(t-s)} \bke{1+(t-s)^{-\beta}} \bke{ \|u(s)\|_{L^{r}(\Omega)}+  \| u v(s)\|_{L^{r}(\Omega)} +\|\varphi(s)  \|_{L^{r}(\Omega)}}ds
\\
&\le C+C\|u  \nabla \log v\|_{L^{\infty}(\Omega\times(0,\infty))}\int_{0}^{t}e^{-(t-s)}\bke{1+(t-s)^{-\beta-\frac{1}{2}}} ds
\\
&\quad+C \bke{ \|u\|_{L^{\infty}(\Omega\times(0,\infty))}+  \| u v\|_{L^{\infty}(\Omega\times(0,\infty))} +\|\varphi   \|_{L^{\infty}(\Omega\times(0,\infty))}}\int_{0}^{t}e^{-(t-s)} \bke{1+(t-s)^{-\beta}}ds
\\
&\le C,
\end{align*}
where $C>0$ is a   constant   independent of $t$. Similarly, it follows from the representation  formula of 
$
v$ that
\begin{align*}
&\|A^{\beta}v(t)\|_{L^{r}(\Omega)}
\\
&\le \|A^{\beta}e^{t(\Delta-1)}v_0\|_{L^{r}(\Omega)}+\int_{0}^{t}\|A^{\beta}e^{(t-s)(\Delta-1)}\bke{uv^{\lambda}+\psi}(s)\|_{L^{r}(\Omega)}ds
\\
&\le C\|v_0\|_{W^{1,r}(\Omega)}+C\int_{0}^{t}e^{-(t-s)}\bke{1+(t-s)^{-\beta}}\bke{  \| uv^{\lambda}(s)\|_{L^{r}(\Omega)}+\|\psi(s) \|_{L^{r}(\Omega)}}ds
\\
&\le C +C\bke{  \| uv^{\lambda}\|_{L^{\infty}(\Omega\times(0,\infty))}+\|\psi \|_{L^{\infty}(\Omega\times(0,\infty))}}\int_{0}^{t}e^{-(t-s)}\bke{1+(t-s)^{-\beta}}ds
\\
&\le C,
\end{align*}
where $C>0$ is a   constant   independent of $t$. 
Thus, \eqref{LEM101} is obtained.

Next, we show \eqref{LEM102}--\eqref{LEM103}. Take sufficiently small $\gamma\in(0,1)$ such that $\beta(1-\gamma)>\frac{d}{2r}$.
Then, 
using $D(A^{\beta(1-\gamma)})\hookrightarrow L^{\infty}( \Omega )$, and the fractional interpolation inequality(see, e.g. \cite[p. 28]{H81}),   we have
\[
\|u\|_{L^{\infty}(\Omega)}\le C\|A^{\beta(1-\gamma)}u\|_{L^{r}(\Omega)}\le C\|A^{\beta}u\|_{L^{r}(\Omega)}^{1-\gamma}\|u\|_{L^{r}(\Omega)}^{\gamma}.
\]
It follows by \eqref{LEM101} and the interpolation inequality $\|u\|_{L^{r}(\Omega)}\le \|u\|_{L^{1}(\Omega)}^{\frac{1}{r}}\|u\|_{L^{\infty}(\Omega)}^{1-\frac{1}{r}}$ that
\[
\|u\|_{L^{\infty}(\Omega)}^{1-(1-\frac{1}{r})\gamma}\le C \|u\|_{L^{1}(\Omega)}^{\frac{\gamma}{r}}.
\]
Thus, \eqref{LEM102} is obtained. Next, we   note from  \eqref{VINFMODEL} and the standard elliptic regularity theory
 that $v_{\infty}\in W^{2,p}(\Omega)$ for any finite $p>1$. Using the same embedding relation and the   fractional interpolation inequality as above, due to  
\[
\|v(\cdot,t)-v_{\infty}\|_{L^{r}(\Omega)}\le \|v(\cdot,t)-v_{\infty}\|_{L^{2}(\Omega)}^{\frac{2}{r}}\|v(\cdot,t)-v_{\infty}\|_{L^{\infty}(\Omega)}^{1-\frac{2}{r}},
\]
we have 
\begin{align*}
&\|v(\cdot,t)-v_{\infty}\|_{L^{\infty}(\Omega)}
\\
&\le C\|A^{\beta(1-\gamma)}(v(\cdot,t)-v_{\infty})\|_{L^{r}(\Omega)}
\\
&\le C\|A^{\beta}(v(\cdot,t)-v_{\infty})\|_{L^{r}(\Omega)}^{1-\gamma}\|v(\cdot,t)-v_{\infty}\|_{L^{r}(\Omega)}^{\gamma}
\\
&\le C(\|A^{\beta}v(\cdot,t)\|_{L^{r}(\Omega)}+\|A^{\beta}v_{\infty}\|_{L^{r}(\Omega)})^{1-\gamma}\|v(\cdot,t)-v_{\infty}\|_{L^{2}(\Omega)}^{\frac{2\gamma}{r}}\|v(\cdot,t)-v_{\infty}\|_{L^{\infty}(\Omega)}^{(\gamma-\frac{2\gamma}{r})}
\\
&\le C\|v(\cdot,t)-v_{\infty}\|_{L^{2}(\Omega)}^{\frac{2\gamma}{r}}\|v(\cdot,t)-v_{\infty}\|_{L^{\infty}(\Omega)}^{(1-\frac{2}{r})\gamma}
\end{align*}
and thus, 
\[
\|v(\cdot,t)-v_{\infty}\|_{L^{\infty}(\Omega)}^{1-(1-\frac{2}{r})\gamma}\le C\|v(\cdot,t)-v_{\infty}\|_{L^{2}(\Omega)}^{\frac{2\gamma}{r}},
\]
where $C>0$ is a   constant   independent of $t$. 
Namely, \eqref{LEM103} is obtained. 
This completes the proof.
 \end{proof}
 
Now, we are ready to prove Theorem~\ref{THM3}.
Due to \eqref{ASSPHIPSI} and   $-\sigma uv$ term in $u$-equation,     $(u,v)$ converges to $(0,v_{\infty})$ as times goes to infinity.
\begin{pfthm4} 
First, we  show the convergence   $u\rightarrow0$. Integrating the first equation of \ref{MODEL0} over $\Omega$ yields
\[
\frac{d}{dt}\int_{\Omega}u+\sigma\int_{\Omega}uv=\int_{\Omega}\varphi.
\]
Using $v\ge  \eta_{0}$,    we have
\[
  \frac{d}{dt}\int_{\Omega}u+  \eta_{0}\sigma\int_{\Omega}u \le  \int_{\Omega}\varphi
\]
and thus,  
\[
\frac{d}{dt}\bke{ e^{  \eta_{0}\sigma t}   \int_{\Omega}u            } \le e^{  \eta_{0}\sigma t}\int_{\Omega}\varphi.
\]
Integrating this with respect to the temporal variable over   $(\tau,2\tau)$ for $\tau>0$ yields
\begin{align*}
e^{  \eta_{0}\sigma 2\tau}   \int_{\Omega}u(\cdot,2\tau)&\le e^{  \eta_{0}\sigma \tau}  \int_{\Omega}u(\cdot,\tau)+ \int_{\tau}^{2\tau}e^{  \eta_{0}\sigma t}\int_{\Omega}\varphi(\cdot,t)dt 
\\
&\le e^{ \eta_{0}\sigma \tau}  \int_{\Omega}u(\cdot,\tau)+ e^{  \eta_{0}\sigma 2\tau}\int_{\tau}^{2\tau}\int_{\Omega}\varphi(\cdot,t)dt
\end{align*}
Then,
\[
 \int_{\Omega}u(\cdot,2\tau)\le e^{-  \eta_{0}\sigma \tau}  \int_{\Omega}u(\cdot,\tau)+ \int_{\tau}^{\infty} \int_{\Omega}\varphi(\cdot,t)dt.
\]
By  \eqref{THM1_2}  and the spatio-temporal $L^{1}$-bound of $\varphi$,  we have that for any given $\varepsilon_0>0$, there exists $\tau>0$ satisfying
\[
 \int_{\Omega}u(\cdot,t)\le \varepsilon_0\qquad \mbox{ for all } t\ge\tau.
\]
Hence, $\int_{\Omega}u$ converges to $0$ as time goes to infinity and thus,
by \eqref{LEM102}, we obtain the desired convergence
\begin{equation}\label{UINFCONV}
\|u(\cdot,t)\|_{L^{\infty}(\Omega)}\rightarrow 0\qquad \mbox{ as } t\rightarrow\infty.
\end{equation}
 
 Next, we  show the convergence   $v\rightarrow v_{\infty}$.
 Note from  \eqref{VINFMODEL} and the standard elliptic regularity theory
 that $v_{\infty}\in W^{2,p}(\Omega)$ for any finite $p>1$.
Using \eqref{MODEL0} and \eqref{VINFMODEL}, we can derive
\begin{align*}
\frac{1}{2}\frac{d}{dt}&\int_{\Omega}|v(\cdot,t)-v_{\infty}|^{2}+\int_{\Omega}|\nabla(v(\cdot,t)-v_{\infty})|^{2}+\int_{\Omega}|v(\cdot,t)-v_{\infty}|^{2}\\
&=\int_{\Omega}(v(\cdot,t)-v_{\infty})uv^{\lambda}+\int_{\Omega}(v(\cdot,t)-v_{\infty})(\psi-\psi_{\infty})
\end{align*}
which entails by Young's inequality that
\[
\frac{1}{2}\frac{d}{dt}\int_{\Omega}|v(\cdot,t)-v_{\infty}|^{2}+\frac{1}{2}\int_{\Omega}|v(\cdot,t)-v_{\infty}|^{2}
\le \int_{\Omega}u^{2}v^{2\lambda}+\int_{\Omega}|\psi-\psi_{\infty}|^{2}.
\]
By \eqref{THM1_2} and a direct computation, we obtain
\[
\frac{d}{dt}\bke{e^{t}\int_{\Omega}|v(\cdot,t)-v_{\infty}|^{2}}\le Ce^{t}\|u(\cdot,t)\|_{L^{\infty}}^{2}+2e^{t}\int_{\Omega}|\psi-\psi_{\infty}|^{2}
\]
for some constant $C>0$ independent of $t$.
Integrating this with respect to the temporal variable over   $(\tau,2\tau)$ with $\tau>0$ yields
\[
 e^{2\tau}\int_{\Omega}|v(\cdot,2\tau)-v_{\infty}|^{2}
 \]
 \[
  \le  e^{\tau}\int_{\Omega}|v(\cdot,\tau)-v_{\infty}|^{2} +C(e^{2\tau}-e^{\tau})\sup_{\tau\le t\le 2\tau}\|u(\cdot,t)\|_{L^{\infty}}^{2}+2e^{2\tau}\int_{\tau}^{2\tau}\int_{\Omega}|\psi-\psi_{\infty}|^{2}
\]
and therefore,
\[
 \int_{\Omega}|v(\cdot,2\tau)-v_{\infty}|^{2}
 \]
 \[
  \le  e^{-\tau}\int_{\Omega}|v(\cdot,\tau)-v_{\infty}|^{2} +C\sup_{t\ge\tau }\|u(\cdot,t)\|_{L^{\infty}}^{2}+2 \int_{\tau}^{\infty}\int_{\Omega}|\psi-\psi_{\infty}|^{2}.
\]
Due to  \eqref{THM1_2}, \eqref{ASSPHIPSI},  and \eqref{UINFCONV}, for any given $\varepsilon_0>0$, there exists $\tau>0$ satisfying
\[
 \int_{\Omega}|v(\cdot,t)-v_{\infty}|^{2}\le \varepsilon_0\qquad \mbox{ for all } t\ge\tau.
\]
Hence, $L^2(\Omega)$-norm of $ v-v_{\infty}$ approaches $0$ as time tends to infinity and  thus, due to \eqref{LEM103}, we obtain the desired convergence
\[
\|v(\cdot,t)-v_{\infty}\|_{L^\infty(\Omega)}\rightarrow 0\qquad  \mbox{ as } t\rightarrow\infty.
\]
 This completes the proof.
\end{pfthm4}
 
\section*{Acknowledgement}
J. Ahn is supported by the Dongguk University Research Fund of 2020.
K. Kang is partially supported by NRF-2019R1A2C1084685 and NRF-2015R1A5A1009350.
J. Lee is supported by Samsung Science and Technology Foundation under Project Number SSTF-BA1701-05.

\end{document}